\documentclass[11pt]{amsart}

\usepackage{amssymb} 
\usepackage{amsmath,amsthm}
\usepackage{graphicx}
\usepackage{tikz}

\usepackage{caption}
\usepackage{subcaption}

\theoremstyle{plain}
\newtheorem{thm}{Theorem}[section]
\newtheorem{lem}[thm]{Lemma}

\theoremstyle{remark}

\numberwithin{equation}{section}

\begin{document}

\title[Self-similar solutions to the MCF in the Minkowski plane]{Self-similar solutions to the mean curvature flow in the Minkowski plane $\mathbf R^{1,1}$}

\author{Hoeskuldur P. Halldorsson}
\address{MIT, Department of Mathematics, 77 Massachusetts Avenue, Cambridge, MA 02139-4307.}
\email{hph@math.mit.edu}

\subjclass[2010]{Primary 53C44. Secondary 53B30}

\date{} 


\begin{abstract}
We introduce the mean curvature flow of curves in the Minkowski plane $\mathbf R^{1,1}$ and give a classification of all the self-similar solutions. In addition, we describe five other exact solutions to the flow.
\end{abstract}

\maketitle

\section{Introduction}
Minkowski space $\mathbf R ^{n,1}$ is the linear space $\mathbf R^{n+1}$ endowed with the Minkowski metric
\begin{equation*}
\langle \,, \rangle = dx_1^2 + \cdots + dx_n^2 - dx_{n+1}^2.
\end{equation*}
The mean curvature flow (MCF) of immersed hypersurfaces in Minkowski space is defined as follows: Let $M^n$ be an $n$-dimensional manifold and consider a family of smooth immersions $X_t = X(\cdot,t) : M^n \rightarrow \mathbf R ^{n,1}$ for $t \in I$. Write $M_t = X_t(M^n)$. The family of hypersurfaces $(M_t)_{t\in I}$ is said to \emph{evolve by mean curvature} if
\begin{equation*}
\frac{\partial X}{\partial t}(p,t) = \mathbf H(p,t)
\end{equation*}
for $p \in M^n$ and $t \in I$. Here $\mathbf H$ is the mean curvature vector of $M_t$, which is normal to the surface and satisfies $\langle \mathbf H,\mathbf n \rangle = -\text{div}_{M_t}\mathbf n$ where $\mathbf n$ is a unit normal field. This flow has been studied in \cite{aar06,ding12,eck93,eck97,eck03,eckhuisk91,huiskyau96,jian06,jianjuliu10}.

A natural question to ask is whether there exist any self-similar solutions to the flow, i.e., hypersurfaces which move under a combination of dilations and isometries of the Minkowski space. The most basic example is the hyperboloid $x_{n+1}^2 = x_1^2 + \cdots + x_n^2 + 2nt$, which is a time-like contracting solution for $t<0$, and a space-like expanding solution for $t>0$. There are also examples of space-like hypersurfaces translating with constant speed along the $x_{n+1}$-axis, both rotationally symmetric in \cite{eck97,jian06}, and more general in \cite{ding12}. Of course, we also have the trivial examples of space-like maximal and time-like minimal surfaces, which are not affected by the flow due to their vanishing mean curvature.

In this paper, we consider the case $n=1$, the MCF of curves in the Minkowski plane $\mathbf R^{1,1}$. Our main result is the following classification of all self-similar solutions to the flow. More details are provided in later sections of the paper.
\begin{thm} Up to rescalings and isometries of the Minkowski plane, the following list contains all self-similar solutions to the mean curvature flow of space-like curves in the Minkowski plane:
\begin{itemize}
\item Translation: Three curves,
\begin{equation*}
\cosh x = e^{y-t}, \quad \sinh y = e^{t-x}, \quad x+y = e^{x-y}+t,
\end{equation*}
which translate along the $y$-axis, $x$-axis and the line $y=x$, respectively. See Figure \ref{transi}.
\item Expansion: Six types of curves, including the expanding hyperbola
\begin{equation*}
y^2 = x^2 + 2t, \quad t>0.
\end{equation*}
See Figures \ref{expansib}, \ref{expansic}.
\item Contraction: One type of curves. See Figure \ref{contrasib}.
\item Hyperbolic rotation: Three types of curves. See Figure \ref{rotib}.
\item Hyperbolic rotation and expansion: Eighteen types of curves, including the exact solutions
\begin{equation*}
x+y = 2t \tan(x-y), \quad x+y =-2t \coth(x-y), \quad t>0.
\end{equation*}
See Figures \ref{ut}, \ref{rotexpansi1b},  \ref{rotexpansi1c},  \ref{rotexpansi2b}, \ref{rotexpansi2c}. 
\item Hyperbolic rotation and contraction: Seven types of cures, including the exact solution
\begin{equation*}
x+y = -2t\tanh(x-y), \quad t < 0.
\end{equation*}
See Figures \ref{inn}, \ref{rotcontrasi1b}, \ref{rotcontrasi2b}.
\item Hyperbolic rotation, expansion and translation: Five types of curves. See Figure \ref{allir1}.
\item Hyperbolic rotation, contraction and translation: Three types of curves. See Figure \ref{allir2}.
\end{itemize}
Reflecting these curves across the line $y=x$ gives all time-like self-similar solutions, with the direction of $t$ reversed.
\end{thm}

The corresponding classification in the Euclidian plane was previously done by the author in \cite{hph1}. There are a few notable differences between the flows in the two planes. In the Euclidean plane, the flow reduces the length of any simple closed curve, hence it is usually called the curve shortening flow. In the Minkowski plane, however, the flow is not defined for simple closed curves, since the curvature blows up at light-like points. It is only defined for space-like and time-like curves but these can have finite Minkowski-length (without having endpoints). In this paper, we have examples of both space-like curves that decrease in length and that increase in length under the flow, so the flow is neither purely shortening nor lengthening. This is illustrated in Figure \ref{lengdarhasar}.

Many of the curves in this paper have Minkowski-finite ends where the curvature blows up. Therefore, the maximum principle does not apply to them. In fact, we have examples of curves that are initially disjoint but then intersect under the flow. We also have examples of non-uniqueness of the flow, i.e., different solutions starting at the same curve. In comparison, all the curves which arose in the classification in the Euclidean plane have bounded curvature and hence do not show this behaviour.

For various reasons, the classification problem is simpler in the Euclidean plane than in the Minkowski plane. In each plane, the problem of finding curves which move under a (hyperbolic) rotation and a dilation can be reduced to the study of a two-dimensional system of ODEs, such that each curve corresponds to a trajectory in the phase plane of the system. 
In the Minkowski plane, when the dilation is an expansion, the system has two saddle points. As a result, there are many different types of trajectories in the phase plane and hence many different types of curves. In the Euclidean plane, however, the system only has a sink and a source (when the dilation is a contraction), which does not result in as many different types of trajectories. Also, in the Minkowski plane, the self-similar solutions consisting of a hyperbolic rotation and a dilation show very different behaviour depending on whether $a^2<b^2$, $a^2 = b^2$ or $a^2 > b^2$, the constants $a$ and $b$ denoting the initial speeds of the hyperbolic rotation and the dilation, respectively. No trichotomy like this is present in the Euclidean plane.

Another reason for the simpler classification in the Euclidean plane is that
any self-similar solution to the flow consisting of translation and either rotation, dilation or both, can actually be described without the translation by moving the origin to a different location. In the Minkowski plane, this simplification can only be made in the case $a^2 \neq b^2$. As a result, the classification contains self-similar solutions combining all three motions, i.e., hyperbolic rotation, dilation and translation. 

The paper is structured as follows. Section \ref{hyperbolic} covers basic properties of the Minkowski plane and introduces the hyperbolic numbers and the diagonal basis, both of which are well suited for making calculations in the Minkowski plane. In Section \ref{curves} we discuss curves in the Minkowski plane, derive the Frenet formulas, and show how the MCF is equivalent to a few different PDEs. Section \ref{selfsim} is devoted to self-similar solutions. First we find all possible self-similar motions that can arise as solutions to the flow and derive the equation the corresponding curves must satisfy. We then rewrite the curve equation as two separate second order ODEs and as a two-dimensional system of ODEs.  These three forms make it easier to find and describe the curves, which we do in Sections \ref{translating} through \ref{screwdilatingtranslating}, one self-similar motion at  a time.

In Section \ref{other} we derive five exact solutions that are not self-similar; namely,
\begin{equation*}
\begin{aligned}
&\cosh x = e^{-t} \cosh y, \quad t >0, \\
&\cosh x = e^{-t} \sinh y, \quad t \in \mathbf R, \\
&\sinh x = e^{-t} \sinh y, \quad t<0,\\
&\sin y = e^{-t} \sin x, \quad t >0, \\
&\tanh (x+y) = \tan (x-y) \tan 2t, \quad 0<t<\tfrac{\pi}{4},
\end{aligned}
\end{equation*}
as shown in Figures \ref{flosi} and \ref{glaenyr}.
Finally, an appendix discusses how one finds all curves in the Minkowski plane that are invariant under some self-similar motion. We include this to explain why one curve appears in two categories in the classification of the self-similar solutions.

\section{Minkowski plane and hyperbolic numbers}
\label{hyperbolic}

The Minkowski plane $\mathbf R ^{1,1}$  is just $\mathbf R ^2$ endowed with the non-degenerate bilinear symmetric form $\langle (x_1,y_1), (x_2,y_2)\rangle = x_1x_2 - y_1y_2$, which is called the \emph{Minkowski metric}. A nonzero vector $z$ is called \emph{space-like} if $\langle z, z \rangle > 0$, \emph{time-like} if $\langle z, z \rangle < 0$, and \emph{light-like} if $\langle z, z \rangle = 0$. The \emph{Minkowski norm} is defined as $||z|| = \sqrt{|\langle z,z \rangle |}$, and two vectors $z_1$ and $z_2$ are said to be \emph{orthogonal} if $\langle z_1, z_2 \rangle = 0$.

In this paper, we will use the language of hyperbolic numbers (also called split-complex numbers), which are well suited for making calculations in the Minkowski plane, much like the complex numbers are useful when dealing with the Euclidean plane.
We identify the point $(x,y)$ with the \emph{hyperbolic number} $x+hy$, which is just an ordered pair of real numbers, with addition and multiplication defined as follows:
\begin{equation*}
\begin{aligned}
&(x_1+hy_1) + (x_2 + hy_2) = (x_1+x_2) +  h(y_1+y_2)\\
&(x_1+hy_1)\cdot (x_2+hy_2) = (x_1x_2+y_1y_2) +h(x_1y_2+x_2y_1).
\end{aligned}
\end{equation*}
In particular, $h^2=+1$ as opposed to $i^2=-1$ for the complex numbers.
The hyperbolic numbers form a commutative associative algebra which is isomorphic to the quotient $\mathbf R[x]/(x^2-1)$, and to the algebra of symmetric $2 \times 2$ matrices with equal diagonal elements, where 
\begin{equation*}
x+hy \leftrightarrow \left(
\begin{array}{cc}
x & y \\
y & x
\end{array}
\right).
\end{equation*}
The basic properties of the hyperbolic numbers are the same as those of the complex numbers, with a few notable exceptions.

\begin{figure}
\centering
\begin{tikzpicture}[scale=0.75]
\draw [->] (0,-4) -- (0,4);
\draw [->] (-4,0) -- (4,0);
\draw [->] (-4,-4) -- (4,4);
\draw [->] (-4,4) -- (4,-4);
\draw [fill] (0,2) circle [radius=0.04];
\draw [fill] (2,0) circle [radius=0.04];
\draw [fill] (1,1) circle [radius=0.04];
\draw [fill] (1,-1) circle [radius=0.04];
\draw [fill] (2.675, 1.776) circle [radius=0.04];
\draw [fill] (1.776,2.675) circle [radius=0.04];
\draw [fill] (-2.675, -1.776) circle [radius=0.04];
\draw [fill] (-1.776,-2.675) circle [radius=0.04];
\node [right] at (4,0) {$x$};
\node [above] at (0,4) {$y$};
\node [above right] at (4,4) {$\xi$};
\node [below right] at (4,-4) {$\eta$};
\node [below left] at (0,2) {$h$};
\node [below left] at (2,0) {$1$};
\node [above] at (1,1) {$h_+$};
\node [below] at (1,-1) {$h_-$};
\node [right] at (2.675, 1.776) {$e^{h\theta}$};
\node [above] at (1.776, 2.675) {$he^{h\theta}$};
\node [left] at (-2.675, -1.776) {-$e^{h\theta}$};
\node [below] at (-1.776, -2.675) {-$he^{h\theta}$};
\draw[domain=-3.5:3.5,smooth,variable=\t]
plot ({2*sqrt(\t*\t*0.25+1)},{\t});
\draw[domain=-3.5:3.5,smooth,variable=\t]
plot ({\t},{2*sqrt(\t*\t*0.25+1)});
\draw[domain=-3.5:3.5,smooth,variable=\t]
plot ({-\t},{-2*sqrt(\t*\t*0.25+1)});
\draw[domain=-3.5:3.5,smooth,variable=\t]
plot ({-2*sqrt(\t*\t*0.25+1)},{-\t});
\end{tikzpicture}
\caption{The hyperbolic number plane.}
\end{figure}
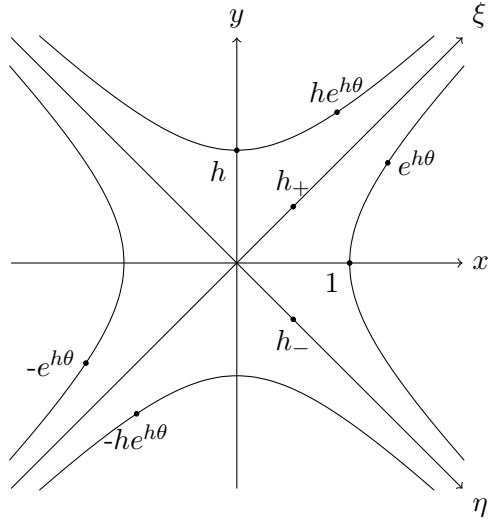

The \emph{hyperbolic conjugate} of $z = x + hy$ is $\bar z = x - hy$, and the \emph{hyperbolic modulus} is defined as $|z| = \sqrt{|z\bar z|}$. They satisfy the usual properties
\begin{equation*}
\overline{z_1+z_2} = \overline {z_1} + \overline {z_2}, \quad \overline{z_1z_2} = \overline{z_1}\,\overline{z_2}, \quad
\overline{\overline{z}} = z, \quad
|z_1z_2| = |z_1||z_2|.
\end{equation*}
Note that
\begin{equation*}
\overline{z_1}z_2 = (x_1x_2-y_1y_2) +h(x_1y_2-x_2y_1)
\end{equation*}
and if we take the real part, we recover the Minkowski metric. In particular, $\bar z z = x^2-y^2 =\langle z, z\rangle $, so the hyperbolic modulus coincides with the Minkowski norm.

Now, $h (x+hy) = y + hx$, so multiplication by $h$ corresponds to reflection across the line $y=x$. Moreover, $z$ and $hz$ are always orthogonal, just like $z$ and $iz$ in the complex plane. However, $\langle hz_1,hz_2\rangle = - \langle z_1,z_2\rangle$.

Perhaps the biggest difference between the hyperbolic numbers and the complex numbers is that the former do not form a field. The reason is that $(1+ h)(1- h) = 0$, so all points on the two lines $y=x$ and $y=-x$ are zero-divisors and hence do not have a multiplicative inverse. These are exactly the points with hyperbolic modulus zero, i.e., light-like points. For other points, the multiplicative inverse is given by $z^{-1} = \frac{\bar z}{\overline z z} = \frac{x-hy}{x^2-y^2}$.

It is often convenient to use the conjugate light-like points $h_+ = \frac{1+ h}{2}$ and $h_- = \frac{1- h}{2}$ as an alternate basis for the hyperbolic numbers, called the \emph{diagonal basis}. In this paper, we will use the notation
\begin{equation*}
x + hy = \xi h_+ + \eta h_- =: (\xi,\eta) 
\end{equation*}
where $\xi = x+y$ and $\eta = x-y$, and refer to the light-like lines $y=x$ and $y=-x$ as the $\xi$ and $\eta$-axes, respectively. Since $h_+^2 = h_+$, $h_-^2 = h_-$ and $h_+h_-=0$, the basic operations now take the simple form
\begin{equation*}
\begin{aligned}
&(\xi_1,\eta_1)+(\xi_2,\eta_2) = (\xi_1+\xi_2,\eta_1+\eta_2) \\
&(\xi_1,\eta_1)\cdot (\xi_2,\eta_2) = (\xi_1\xi_2,\eta_1\eta_2)\\
&\overline{(\xi,\eta)} =  (\eta,\xi)\\
&\langle (\xi,\eta),(\xi,\eta)\rangle = \xi\eta.\\
\end{aligned}
\end{equation*}
In particular, we see that the hyperbolic numbers are isomorphic to $\mathbf R \oplus \mathbf R$ with pairwise addition and multiplication.

The set of points with $\langle z, z \rangle =  1$ is the unit hyperbola $x^2-y^2=  1$, i.e., $\xi\eta =  1$.
Note that by plugging $h\theta$ into the power series of the exponential function, we get the following hyperbolic analogue of Euler's formula
\begin{equation*}
e^{h\theta} = \cosh \theta + h \sinh \theta = (e^\theta,e^{-\theta}).
\end{equation*}
These points lie on the right arm of the unit hyperbola and they form a one-parameter group, much like the unit circle in the complex plane. The parameter $\theta$ is called the \emph{hyperbolic angle}.

Maps of the form $z \mapsto e^{h\theta}z$ preserve the hyperbolic modulus and are called hyperbolic rotations (also known as Lorentz boosts or squeeze maps).  If we include translations and reflections across the $x$ and $y$-axes, we have the full group of isometries in the Minkowski plane.

For further background on the hyperbolic number plane we refer to \cite{cat11,sobcz95}.

\section{Curves in the Minkowski plane}
\label{curves}

\begin{figure}[t]
\centering 
\includegraphics{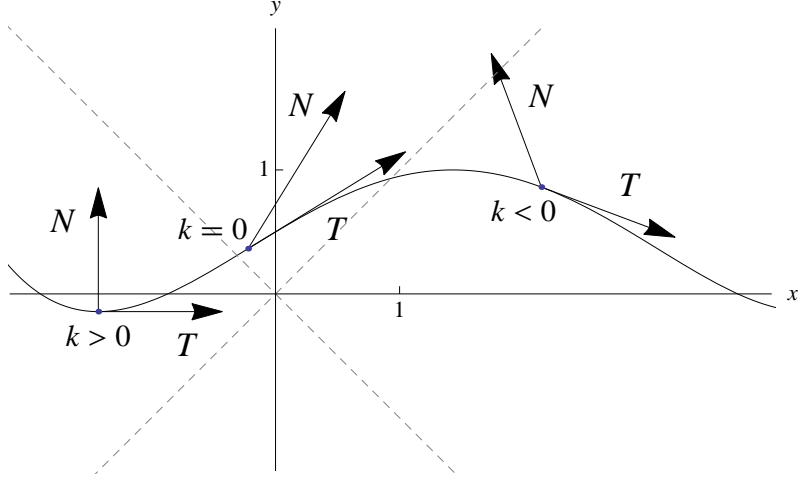} 
\caption{A space-like curve in the Minkowski plane.} 
\label{Mcurve} 
\end{figure}

Let $X:I \rightarrow \mathbf R ^{1,1}$ be a regular curve. At each point on the curve, the tangent vector $X_u$ is either space-like, light-like or time-like. The curvature blows up at points where $X_u$ is light-like, so we will only look at curves where it is everywhere space-like or time-like. Those curves are called space-like and time-like respectively and are always embedded. Note that reflection across the $\xi$-axis maps space-like curves to time-like and vice versa. 

The \emph{Minkowski arc-length} parameter is defined as
\begin{equation*}
ds = \sqrt{|\langle X_u,X_u \rangle|} du
\end{equation*}
and the \emph{unit tangent vector} is
\begin{equation*}
T = X_s = \frac{1}{ \sqrt{|\langle X_u,X_u \rangle|}}X_u.
\end{equation*}
We choose the \emph{unit normal} $N = hT$, i.e., the vector obtained by reflecting $T$ across the $\xi$-axis. Since $\langle T, T \rangle = \epsilon = \pm 1$ is constant, $T_s$ is parallel to $N$, and  we define the \emph{signed curvature} $k$ by the equation $T_s = kN$. Then we also have $N_s = kT$, so the Frenet formulas take the form
\begin{equation*}
\begin{aligned}
T_s &= k N \\
N_s &= k T.
\end{aligned}
\end{equation*}
By definition, the \emph{mean curvature vector} of $X$ is then given by $\mathbf H = \epsilon  X_{ss} = \epsilon T_s =\epsilon k N$, so we always have $\langle \mathbf H, N \rangle = -k$. Note that many authors choose $N$ such that $\{T,N\}$ is positively oriented, resulting in some different signs in the Frenet formulas above.

If we reflect a time-like curve across the $\xi$-axis, it not only becomes space-like but the mean curvature vector also changes its direction (because of the $\epsilon$ factor). This means the direction of the MCF of the curve is reversed, but besides that the flow is the same. For this reason, it suffices to look at the flow of space-like curves, even though we will use time-like curves to simplify some arguments.

We will parametrize our space-like curves such that $T$ lies in the right arm of the unit hyperbola $x^2-y^2=1$. Then $T=e^{h\theta}$, for some real number $\theta$ called the \emph{hyperbolic tangent angle}. Note that $T_s = \theta_s h e^{h \theta} = kN$, so $\theta$ satisfies $\theta_s = k$. The upper arm of the unit hyperbola $x^2-y^2 = -1$ (or its translates) is the unique space-like curve with constant curvature $k=1$. A curve such that $k>0$ (or $k<0$) everywhere is called \emph{convex}.


Besides the arc-length parametrization, every space-like curve in the Minkowski plane has two other natural parametrizations that we will use. The curve can globally be viewed as a graph where either $y$ is a function of $x$ satisfying $|y'(x)|<1$, or $\xi$ is an increasing function of $\eta$.
In the former case, the parametrization is $X(x) = x + hy(x)$ and direct calculations yield
\begin{equation}
\label{yjofnur}
T = \frac{1+hy'(x)}{(1-y'(x)^2)^{\frac 1 2}}, \quad N = \frac{y'(x) + h}{(1-y'(x)^2)^{\frac1 2}}, \quad
k = \frac{y''(x)}{(1-y'(x)^2)^{\frac 3 2}}.
\end{equation}
If we think of $y$ as also being a function of the time variable $t$, we have a solution to the MCF (up to tangential diffeomorphisms) if and only if $y$ satisfies the parabolic PDE
\begin{equation}
\label{yPDE}
y_t = \frac{y_{xx}}{1-y_x^2}.
\end{equation}
In the Euclidean plane, the PDE corresponding to the curve shortening flow is
\begin{equation}
\label{yPDE3}
y_t = \frac{y_{xx}}{1+y_x^2}.
\end{equation}
The following lemma shows how one can transform certain solutions in the Euclidean plane into solutions in the Minkowski plane. Examples of this are given in Section \ref{other}.
\begin{lem}
\label{transformers}
If $y(x,t)$ is a solution to \eqref{yPDE3} which is analytic and even in $x$, then $\hat y(x,t) = y(ix,-t)$ is a solution to \eqref{yPDE}, where $i^2=-1.$ 
\end{lem}

In the case where $\xi$ is an increasing function of $\eta$, the parametrization is $X(\eta) = (\xi(\eta), \eta)$ and direct calculations yield
\begin{equation}
\label{xijofnur}
T = (\xi'(\eta)^{\frac 1 2}, \xi'(\eta)^{-\frac 1 2}), \quad N = (\xi'(\eta)^{\frac 1 2},-\xi'(\eta)^{-\frac 1 2}), \quad k = \frac{\xi''(\eta)}{2\xi'(\eta)^{\frac 3 2}}.
\end{equation}
The parabolic PDE for the flow now takes the simple form
\begin{equation}
\label{xiPDE}
\xi_t = \frac{\xi_{\eta\eta}}{\xi_\eta}.
\end{equation}
Note that if we instead let $|y'(x)| > 1$ or $\xi'(\eta) < 0$, the curve $X$ becomes time-like. Equations \eqref{yjofnur} and \eqref{xijofnur} take on a slightly different form but the PDEs \eqref{yPDE} and \eqref{xiPDE} remain the same. However, now they are only parabolic if we reverse the direction of the time variable, corresponding to what we said before.

It was shown in \cite{gagham86} that for convex curves in the Euclidean plane, the flow is equivalent to the following PDE for the curvature $k$:
\begin{equation}
\label{kPDE3}
k_t = k^2k_{\theta\theta} + k^3,
\end{equation}
where the derivative with respect to $t$ is taken with $\theta$ held fixed. The corresponding PDE in the Minkowski plane is 
\begin{equation}
\label{kPDE2}
k_t = k^2k_{\theta\theta} - k^3,
\end{equation}
which can be seen by carrying out the Minkowski version of the calculations in \cite{gagham86}.
The following lemma gives another method for transforming solutions between the two planes.	
\begin{lem}
\label{transformers2}
If $k(\theta,t)$ is a solution to \eqref{kPDE3} which is analytic and even in $\theta$, then $\hat k(\theta,t) = k(i\theta,-t)$ is a solution to \eqref{kPDE2}. 
\end{lem}

\section{Self-similar motions of curves under the MCF}
\label{selfsim}

Let $X:I \rightarrow \mathbf R^{1,1}$ be a curve. A \emph{self-similar motion} of $X$ is a map $\hat X: I \times J \rightarrow \mathbf R^{1,1}$ of the form
\begin{equation}
\label{hreyfing}
\hat X(u,t) = g(t)e^{hf(t)}X(u) + H(t) .
\end{equation}
Here $J$ is an interval containing $0$ and $f,g: J\rightarrow \mathbf R$ and $H: J \rightarrow \mathbf R^{1,1}$ are differentiable functions s.t.\ $f(0) = 0$, $g(0) = 1$ and $H(0) = 0$, and hence $\hat X(u,0) = X(u)$. The function $f$ determines the hyperbolic rotation, $g$ determines the dilation and $H$ is the translation term.

This motion is the mean curvature flow of $X$ (up to tangential diffeomorphisms) if and only if the equation
\begin{equation*}
\left\langle \frac{\partial \hat X}{\partial t}(u,t),N(u,t) \right\rangle = -k(u,t)
\end{equation*}
holds for all $u \in I$, $t  \in J$. Simple calculations yield that this equation is equivalent to
\begin{equation}
\label{FullJafna}
\begin{aligned}
g^2(t)f'(t)\langle X(u),T(u) \rangle&- g(t)g'(t)\langle X(u),N(u)\rangle   \\
&- g(t)\langle e^{-hf(t)}H'(t),N(u)\rangle = k(u).
\end{aligned}
\end{equation}
By looking at this equation at time $t=0$, we see that $X$ has to satisfy
\begin{equation}
\label{AdalJafna}
a\langle X,T \rangle - b \langle X,N \rangle - \langle C,N \rangle = k,
\end{equation}
where $a=f'(0)$, $b = g'(0)$ and $C = H'(0)$. It turns out that satisfying an equation of this form is also a sufficient condition for $X$ to move in a self-similar manner under the MCF. To see that, we treat separately three cases.\\

{\bf Translation:} First assume $X$ satisfies the pure translation equation
\begin{equation*}
- \langle C,N \rangle = k.
\end{equation*}
Then we can take $H=Ct$, and easily verify that Equation \eqref{FullJafna} is satisfied for all $u \in \mathbf R$, $t\in J$. Thus, under the flow, the curve $X$ translates with constant velocity vector $C$.\\

{\bf Dilation and rotation:}
Now assume $X$ satisfies the screw-dilation equation
\begin{equation}
\label{RotScal}
a\langle X,T \rangle - b \langle X,N \rangle  = k.
\end{equation}
If the functions $g$ and $f$ satisfy $g(t)g'(t) = b$ and $g^2(t)f'(t) = a$ for all $t\in J$, then Equation \eqref{FullJafna} is satisfied for all $u\in I$, $t\in J$. Solving these differential equations with our initial values gives
\begin{equation}
\label{formulagf}
g(t) = \sqrt{2bt+1} \quad \text{and} \quad f(t) = \begin{cases} \frac{a}{2b}\log(2bt+1) & \text{if } b\neq0,\\
at & \text{if } b=0. \end{cases} 
\end{equation}
Therefore, under the MCF the curve $X$ rotates and dilates around the origin as governed by the functions $f$ and $g$.

Note that when $b\neq0$, the screw-dilation factor $ g(t)e^{hf(t)}$ takes the following form in the
diagonal basis:
 \begin{equation*}
 \begin{aligned}
 (g(t)e^{f(t)},g(t)e^{-f(t)}) &=(g(t)^{1+\frac a b},g(t)^{1-\frac{a}{b}})\\
 &= ((2bt+1)^{\frac{b+a}{2b}},(2bt+1)^{\frac{b-a}{2b}}).
 \end{aligned}
 \end{equation*}
In the cases $a=b$ and $a=-b$, the factor is simply $((2bt+1),1) $ and $(1,(2bt+1))$, respectively, so the screw-dilation reduces to a linear dilation in one of the variables $\xi$ and $\eta$, leaving the other one intact.\\

{\bf Dilation, rotation and translation}:
Now assume $X$ satisfies the full Equation \eqref{AdalJafna}, which can also be written in the form
\begin{equation}
\label{AdalJafna2}
-\langle (b+ha)X + C,N\rangle = k.
\end{equation}
In the Euclidean plane, the author showed in \cite{hph1} that we can always get rid of the translation term $C$ by translating $X$ by a fixed vector. The same thing can be done here, but only when $a^2\neq b^2$, since otherwise $b+ha$ is a zero divisor.  If we put $\hat X = X + \frac{C}{b+ha}$, then $\hat X$ satisfies an equation of the form \eqref{RotScal}. Therefore, $X$ screw-dilates around the point $-\frac{C}{b+ha}$.

We are left with the case $a^2=b^2$. By reflecting the curve across the $y$-axis if necessary, we may assume $a=b$.  By translating $X$ along the $\xi$-axis, we can cancel out the $\xi$-component of $C$, so we may assume $C$ only has a $\eta$-component, i.e., $C = (0,c)$.
As before, $f$ and $g$ are given by \eqref{formulagf} and the screw-dilation factor is $((2bt+1), 1)$, so we just have scaling in the variable $\xi$. We want the translation function $H$ to satisfy 
\begin{equation*}
g(t)e^{-hf(t)}H'(t) = C,\quad \text{i.e.,} \quad H'(t) = \frac{e^{hf(t)}}{g(t)}(0,c) = \left(0,\frac{c}{1+2bt}\right),
\end{equation*}
which by integration results in
\begin{equation*}
H(t) = \left(0,\frac{c}{2b}\log(2bt+1)\right).
\end{equation*}
The self-similar motion under the flow therefore consists of scaling in the $\xi$ variable and translation in the $\eta$ variable.\\

Thus the classification of all self-similar solutions to the MCF in the Minkowski plane has been reduced to finding all curves that satisfy Equation \eqref{AdalJafna}. 
To find these curves, we use two different approaches. The straightforward one is to use \eqref{yjofnur} or \eqref{xijofnur} to rewrite \eqref{AdalJafna} as an ODE for $y$ as a function of $x$, or for $\xi$ as a function of $\eta$. The two ODEs are
\begin{equation}
\label{yODE}
y''(x) = (1-y'(x)^2)\left(a(x-y(x)y'(x))-b(xy'(x)-y(x))-(c_1y'(x)-c_2)\right)
\end{equation}
and
\begin{equation*}
\label{xiODE}
\xi''(\eta) = \xi'(\eta)((a+b)\xi(\eta)+(a-b)\eta \xi'(\eta) + d_1-d_2\xi'(\eta))
\end{equation*}
where $C =  c_1 + hc_2 = (d_1,d_2)$. Depending on the values of $a$, $b$ and $C$, one of the ODEs can be easier to analyze than the other. It can be verified that the equations remain the same when the curve $X$ is assumed to be time-like.

A nice property that we see from \eqref{yODE} is the following.
\begin{lem}
All space-like curves satisfying Equation \eqref{AdalJafna} are entire graphs over the $x$-axis.
\end{lem}
\begin{proof}
First notice that every light-like line  $y(x) = x+B$ or $y(x) = -x+B$ is a solution to \eqref{yODE}. Therefore, if we have a solution to \eqref{yODE} which satisfies $|y'(x)|<1$ at some point, it will satisfy that strict inequality everywhere, since if $y'$ reached $-1$ or $1$, we would contradict uniqueness (because the ODE is of order 2). Since $y'$ is bounded, the solution extends uniquely to all of $\mathbf R$.
\end{proof}
A consequence of this is that all the curves are complete as Euclidean curves. However, as we will see below, they are not necessarily complete as Minkowski curves. We also note that for each value of $a$, $b$ and $C$, we have a two-parameter family of curves satisfying \eqref{AdalJafna} (parametrized for example by the values of $y(0)$ and $y'(0)$). However, by identifying curves that are equal up to a hyperbolic rotation, there is actually only a one-parameter family of curves.\\

The other approach to describing the curves is similar to what was done by the author in \cite{hph1,hph2}. It can be used for curves that satisfy Equation \eqref{RotScal} and goes as follows.
We introduce the functions $\tau = \langle X, T \rangle$ and $\nu = \langle X, N \rangle$. They satisfy
\begin{equation}
\label{TauNuDiffur}
\begin{aligned}
 \tau_s &= 1 + k\nu \\
\nu_s &= k\tau
\end{aligned}
\end{equation}
and furthermore,
\begin{equation}
\label{Grunnjafna}
X = (\tau - h \nu)T \quad\text{and}\quad \langle X, X \rangle = \tau^2-\nu^2.
\end{equation}
The Minkowski analogue of Lemma 3.1 in \cite{hph2} is as follows.
\begin{lem}
\label{FerillFenginn}
For every smooth function $\Phi:\mathbf R^2\rightarrow \mathbf R$, point $z_0\in\mathbf R^{1,1}$ and hyperbolic angle $\theta_0 \in \mathbf R$, there is a unique nonextendible space-like curve $X:I \rightarrow \mathbf R^{1,1}$ satisfying the equation $k = \Phi(\tau,\nu)$ and going through $z_0$ with hyperbolic angle $\theta_0$.
\end{lem}
\begin{proof}
Keeping in mind \eqref{TauNuDiffur} and \eqref{Grunnjafna}, we let $\tau, \nu, \theta$ be the unique maximal solution to the ODE system
\begin{equation*}
\label{taunuODE}
\left\{
\begin{aligned}
\tau' &= 1 + \nu\Phi(\tau,\nu) \\
\nu' &= \tau\Phi(\tau,\nu)\\
\theta' &= \Phi(\tau,\nu)
\end{aligned} \right.
\end{equation*}
with initial values $\theta(0) = \theta_0$, $\tau(0)-h\nu(0) = e^{-h\theta_0}z_0$, and then define the curve as
\begin{equation}
\label{Xjafna}
 X = (\tau-h\nu)e^{h\theta}.
\end{equation}
Note that
\begin{equation*}
X' = (\tau'-h\nu'+h\theta'(\tau-h\nu))e^{h\theta} = e^{h\theta},
\end{equation*}
so $X$ is parametrized by Minkowski arc-length with tangent $T = e^{h\theta}$, and hence the curvature $k$ is equal to $\theta' = \Phi(\tau,\nu)$. Finally,
\begin{equation*}\begin{aligned}
\langle X,T \rangle &=\text{Re}(Xe^{-h\theta} )= \tau, \\
\langle X,N \rangle &= \text{Re}(X(-h)e^{-h\theta} )= \nu,
\end{aligned}\end{equation*}
finishing the proof.
\end{proof}

So all curves satisfying Equation \eqref{RotScal} (i.e., $k = a\tau - b\nu$) can be found by solving the system of ODEs
\begin{equation}
\label{taunuODE2}
\left\{
\begin{aligned}
\tau' &= 1 + \nu(a\tau-b\nu) \\
\nu' &= \tau(a\tau-b\nu).
\end{aligned} \right.
\end{equation}
We can also let $l = -b\tau + a\nu$ and study the equivalent system
\begin{equation*}
\left\{
\begin{aligned}
k' &= a + kl \\
l' &= -b + k^2.
\end{aligned} \right.
\end{equation*}
Note that $k+hl=(a-hb)(\tau+h\nu)$. The corresponding Euclidean system was used by the author in \cite{hph1}. From \eqref{Xjafna} we see that
\begin{equation*}
\xi = e^\theta (\tau - \nu) = e^\theta  \frac{k-l}{a+b}, \quad \eta = e^{-\theta}(\tau + \nu) = e^{-\theta}\frac{k+l}{a-b}.
\end{equation*}
Hence, $X$ intersects the light-like axes exactly when the $(\tau,\nu)$ and $(k,l)$-trajectories intersect the diagonals in the respective phase planes. 

Also note that the right hand side of \eqref{taunuODE2} remains the same when $(\tau,\nu)$ is replaced by $(-\tau,-\nu)$. Therefore, $s \mapsto -(\tau(-s),\nu(-s))$ is also a solution to the system, which of course just corresponds to the curve $X$ parametrized backwards (and reflected across the origin to keep $T$ in the right arm of the unit hyperbola). This symmetry will simplify some of our arguments.

Unlike the Euclidean case, the solutions to this ODE system are not necessarily defined on all of $\mathbf R$ and the curvature $k$ can blow up for finite values of $s$. That means the corresponding end of the curve has finite Minkowski arc-length and hence is geodesically incomplete, even though it is always complete as a Euclidean curve, as we showed above.

In the next six sections of the paper, we look separately at different values of $a$, $b$ and $C$, i.e., different self-similar motions, and describe some properties of the corresponding curves.

\section{Translation}
\label{translating}
A curve $X$ translates with constant velocity vector $C$ if and only if it satisfies the translation equation $-\langle C, N \rangle = k$. By scaling the curve, reflecting it across the $x$ and $y$-axes and applying a hyperbolic rotation, it suffices to consider translations with velocity vectors $1$, $h$ and $h_+$. The solutions below are unique up to constant translations in each variable.

\begin{thm} 
There are three translating solutions:
\begin{itemize}
\item $\cosh x = e^{y-t}$,  with $k = \frac{1}{\cos s}$, $-\frac \pi 2 < s < \frac \pi 2$, translates along the $y$-axis,
\item $\sinh y = e^{t-x}$, with $k = \frac{1}{\sinh s}$, $s>0$, translates along the $x$-axis,
\item $\xi = e^{\eta}+t$, with $k = \frac 1 s$, $s>0$, translates along the $\xi$-axis.
\end{itemize}
\end{thm}

The three curves appear in Figure \ref{transi}. Each end is labeled with the corresponding limit of the curvature $k$, a convention we will use throughout the paper.

\begin{figure}[t]
	\centering
	\begin{subfigure}[t]{0.48\textwidth}
		\centering
		\includegraphics{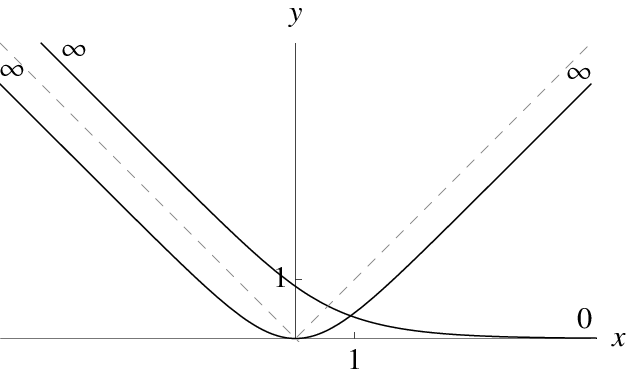}
		\caption{ $ \cosh x = e^{y-t} $ translates along the $y$-axis and $\sinh y =e^{t-x} $ along the $x$-axis.}
		\label{transia}
	\end{subfigure}
	\quad
	\begin{subfigure}[t]{0.48\textwidth}
		\centering
		\includegraphics{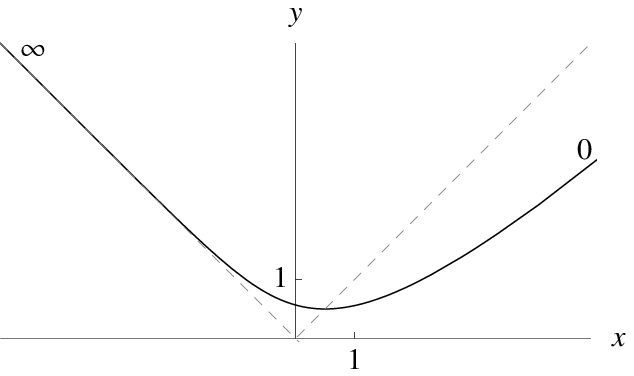}
		\caption{$\xi = e^\eta+t$ translates along the $\xi$-axis.}
		\label{transib}
	\end{subfigure}
	\caption{Translation.}
	\label{transi}
\end{figure}

\begin{proof}
Let's start with $C=h$, i.e., translation with unit speed along the $y$-axis. Then the translation equation results in the ODE
\begin{equation*}
y''(x)=1-y'(x)^2,
\end{equation*}
which can also be seen by putting $y_t = 1$ in Equation \eqref{yPDE}. To find the solution, let $v =y'$. Then $v' = 1-v^2$ and $|v|<1$ (since the curve is space-like), so $v = \tanh x$. Therefore, the curve is $y = \log \cosh x$. It has finite Minkowski-length $\pi$ and curvature given by $k = \frac{1}{\cos s}$, $-\frac \pi 2 < s < \frac \pi 2$. Since the curvature blows up, the maximum principle does not apply to the curve. In fact, as observed by Ecker \cite{eck97}, the translating curve $y = \log \cosh x + t$  is initially below the expanding hyperbola $y = \sqrt{x^2+2t}$, but crosses it at infinity at time $t = \log 2$.

If we instead take $|v| > 1$, we get $v = \coth x$ so $y = \log \sinh x$, $x>0$ (or its reflection across the $y$-axis). This is a time-like curve that translates with unit speed along the $y$-axis. If we reflect it across the $\xi$-axis and then across the $y$-axis (since the first reflection reverses the direction of the flow), we get the space-like curve $x = -\log \sinh y$, which translates with unit speed along the $x$-axis. Alternatively, this curve could be found by taking $C = 1$, which results in the ODE
\begin{equation*}
y''(x) = -y'(x)(1-y'(x)^2).
\end{equation*}
This curve is Minkowski-finite in one direction and Minkowski-infinite in the other, with $k = \frac{1}{\sinh s}$, $s>0$.  In Section \ref{other}, we show how each of these two translating solutions can be obtained by applying a simple transformation to the translating solution in the Euclidean plane, the Grim Reaper $y = \log \cos x + t$.

Finally, let's look at the case $C=h_+$, i.e., translation along the $\xi$-axis. Here it is simpler to work in the diagonal basis, where the ODE takes the form
\begin{equation*}
\xi''(\eta) =  \xi'(\eta).
\end{equation*}
This could also be seen by putting $\xi_t = 1$ in \eqref{xiPDE}. The solution is $\xi = e^{\eta}$. This curve is Minkowski-finite in one direction and Minkowski-infinite in the other, with $k = \frac 1 s$, $s>0$. Another interpretation of the movement of this curve is given in Section \ref{screwdilating} and in the Appendix.
\end{proof}

Note that the three translating curves respectively satisfy
\begin{equation*}
\begin{aligned}
k &= \cosh \theta, \quad \theta \in \mathbf R, \\
k &= - \sinh \theta, \quad \theta < 0, \\
k &= e^{-\theta}, \quad \theta \in \mathbf R.
\end{aligned}
\end{equation*}
They also could have been derived by putting $k_t = 0$ in Equation \eqref{kPDE2}, resulting in the ODE $k_{\theta\theta}-k=0$.

\section{Expansion}
\label{expanding}

A curve expands under the flow if and only if it satisfies Equation \eqref{RotScal} with $a=0$ and $b>0$. By scaling the curve we may assume $b=1$.

\begin{thm}
There are six types of curves which expand under the flow with scaling function $\sqrt{2t+1}$ for $t>-\frac 1 2$, including the expanding hyperbola
\begin{equation*}
y^2 - x^2 =  2t + 1.
\end{equation*}
Each curve is convex, asymptotic to a cone and $k^2e^{\langle X, X \rangle}$ is constant. It comes out of this cone (translated such that its vertex lies at the origin) at $t=-\frac 1 2$, and expands out to infinity, as $t \rightarrow \infty$.
\end{thm}
The curves appear in  Figures \ref{expansib}, \ref{expansic}.

\begin{proof}
Here $k = -\nu$, so the system for $\tau$ and $\nu$ becomes
 \begin{equation*}
\label{Staekkunarhneppi}
\left\{
\begin{aligned}
\tau' &= 1 - \nu^2 \\
\nu'&= -\tau\nu.
\end{aligned}
\right.
\end{equation*}
The phase portrait appears in Figure \ref{expansia}.
The trivial solution $\tau = s$, $\nu=0$, corresponds to $X$ being a straight (space-like) line through the origin, making the expansion vacuous. For the other trajectories, it suffices by symmetry to consider only those in the lower half-plane ($\nu<0$).  There $k>0$, so the corresponding curve $X$ is always convex. Notice that we have a fixed point $(0,-1)$, which corresponds to $X$ being the upper arm of the unit hyperbola $y^2-x^2=1$. It is a saddle point. To find the other trajectories, note that the function $\nu^2e^{\tau^2-\nu^2}=k^2e^{\langle X, X \rangle}$ is a constant $A > 0$. The number and type of the trajectories depend on the value of $A$ as follows:

When $A<e^{-1}$, we have two types of trajectories, both symmetric across the $\nu$-axis. One of them has $\tau \rightarrow \pm\infty$ and $\nu \rightarrow 0$ as $s\rightarrow \pm \infty$. It crosses both lines $\tau = \nu$ and $\tau = -\nu$. The corresponding $X$ crosses each light-like axis and has two Minkoswki-infinite ends with $k\rightarrow0$ on each. The other trajectory has $\tau$ and $\nu$ going to $-\infty$ for a finite $s$. Similarly,  in the backwards direction, $\tau$ and $-\nu$ go to $\infty$ for a finite $s$. It lies completely in the region $-\nu > |\tau|$. The corresponding curve $X$ crosses neither light-like axis and has finite Minkowski-length (which increases under the flow) with $k\rightarrow\infty$ on each end. Both curves appear in Figure \ref{expansib}.

When $A=e^{-1}$, we have two trajectories besides the fixed point $(0,-1)$ (up to reflection across the $\nu$-axis which corresponds to reflecting the curve $X$ across the $y$-axis). One of them has $\tau \rightarrow \infty$, $\nu \rightarrow 0$, as $s\rightarrow \infty$, and $(\tau,\nu) \rightarrow (0,-1)$, as $s \rightarrow -\infty$. It crosses the line $\tau = -\nu$. The corresponding curve $X$ crosses the $\xi$-axis and has two Minkowski-infinite ends with the $k \rightarrow 0$ on one and $k\rightarrow 1$ on the other. The other trajectory has $(\tau,\nu) \rightarrow (0,-1)$, as $s \rightarrow \infty$, and in the backwards direction, $\tau$ and $-\nu$ go to $\infty$ for a finite $s$. It lies completely in the region $-\nu > |\tau|$. The corresponding curve $X$ crosses neither light-like axis, has one Minkowski-infinite end with $k\rightarrow 1$, and one Minkowski-finite end with $k \rightarrow \infty$. Both curves appear in Figure \ref{expansic}.

When $A>e^{-1}$, we have one trajectory type (up to reflection across the $\nu$-axis). It has $\tau \rightarrow \infty$, $\nu \rightarrow 0$, as $s\rightarrow \infty$, and in the other direction, $\tau$ and $-\nu$ go to $\infty$ for a finite $s$. It crosses the line $\tau = -\nu$. The corresponding curve $X$ crosses the $\xi$-axis, has one Minkowski-infinite end with $k\rightarrow 0$, and one Minkowski-finite end with $k\rightarrow \infty$. It can be seen in Figure \ref{expansib}.

\begin{figure}[t]
	\centering
	\begin{subfigure}{0.8\textwidth}
		\centering
		\includegraphics{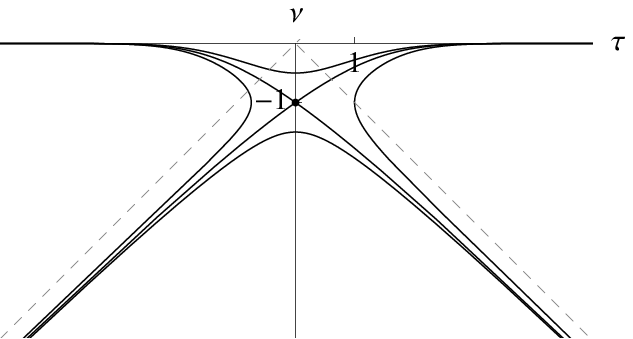}
		\caption{Phase portrait.}
			\vspace{0.1in}
		\label{expansia}
	\end{subfigure} 
	\begin{subfigure}[t]{0.48\textwidth}
		\centering
		\includegraphics{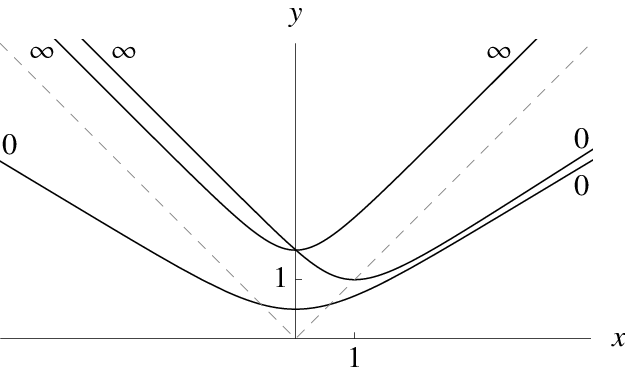}
		\caption{$A < e^{-1}$ and $A > e^{-1}$.}
		\label{expansib}
	\end{subfigure}
	\quad
	\begin{subfigure}[t]{0.48\textwidth}
		\centering
		\includegraphics{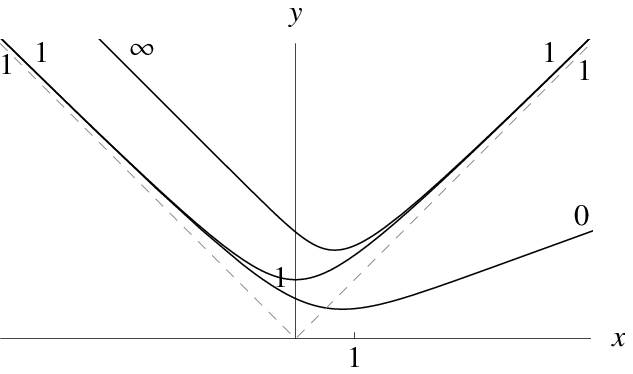}
		\caption{$A = e^{-1}$.}
		\label{expansic}
	\end{subfigure}
	\caption{Expansion, $a=0$, $b=1$.}
	\label{expansi}
\end{figure}

For a more detailed description of the curves, we consider the ODE for $y$ as a function of $x$,
\begin{equation*}
y''(x) = (1-y'(x)^2)(y(x)-xy'(x)).
\end{equation*}
Since $-\nu = k >0$, it is clear that $y$ is convex and $y(0) > 0$. Put $w(x) = y(x)-xy'(x)$. Then $w'(x) = -xy''(x)$, and since $y''(x) > 0$, the positive function $w$ has a global maximum $w(0) = y(0)$, i.e., $0 < y(x)-xy'(x) \leq y(0)$ for all $x$. From this it follows that for $x>0$, $\frac{y(x)}{x}$ is decreasing and $\frac{y(x)-y(0)}{x}$ is increasing, and since $|y'(x)| < 1$, they have the same limit $L \in (y'(0),1]$, as $x\rightarrow \infty$. In particular, $\frac{y(x)-y(0)}{x} < L < \frac{y(x)}{x}$, so the convex function $y(x)-Lx$ lies in the interval $[0,y(0)]$. Hence, it must be decreasing with a limit in that same interval.

So we have shown that $X$ is asymptotic to a straight line with slope $L$, as $x \rightarrow \infty$. If $X$ doesn't cross the $\xi$-axis, then $y(x)>x$ for all $x>0$, so $L=1$. However, if $X$ does cross the $\xi$-axis, then $y(x)$ goes below $x$, so $L<1$.  By symmetry, $X$ is also asymptotic to a straight line, as $x \rightarrow -\infty$. Its slope lies in the interval $[-1,-y'(0))$, and it is $-1$ if and only if $X$ never crosses the $\eta$-axis.

The curve $X$ is therefore asymptotic to a cone. Under the flow, it comes out of this cone (translated such that its vertex lies at the origin) at $t=-\frac 1 2$, and expands out to infinity, as $t \rightarrow \infty$.
\end{proof}

Note that the initial values $y(0) = \alpha$, $y'(0)=0$, correspond to the $(\tau,\nu)$-trajectory going through $(0,-\alpha)$, so varying $\alpha > 0 $ gives all curves with $A < e^{-1}$, in addition to the hyperbola $y = \sqrt{x^2+1}$ when $\alpha = 1$. They are symmetric w.r.t.\ the $y$-axis. Similarly, the curves with $A>e^{-1}$ can be obtained by taking the initial values $y(\alpha)=1$, $y'(\alpha) = 0$, and varying $\alpha>0$.

\section{Contraction}
\label{shrinking}

Now suppose $a=0$, $b<0$, so by scaling the curve we may assume $b=-1$.  

\begin{thm}
There is a one-parameter family of curves which contract under the flow with scaling function $\sqrt{1-2t}$ for $t<\frac 1 2$. Each curve is convex, asymptotic to a cone and symmetric w.r.t.\ the $y$-axis with $k^2e^{-\langle X, X \rangle }$ constant . It comes from infinity as $t \rightarrow -\infty$, and contracts to the union of the positive $\xi$-axis and the negative $\eta$-axis, as $t \rightarrow \frac 1 2$.
\end{thm}
The curves appear in Figure \ref{contrasib}.

\begin{proof}
Here $k =\nu$ so the system of ODEs for $\tau$ and $\nu$ becomes
 \begin{equation*}
\label{Minnkunarhneppi}
\left\{
\begin{aligned}
\tau' &= 1 + \nu^2 \\
\nu'&= \tau \nu.
\end{aligned}
\right.
\end{equation*}
The phase portrait appears in Figure \ref{contrasia}.
There are no fixed points, $\tau$ is increasing and $\nu$ never changes sign. The trivial solution $\tau = s$, $\nu=0$, corresponds to $X$ being a (space-like) line through the origin, making the contraction vacuous. To find the other trajectories, note that the function $\nu^2e^{\nu^2-\tau^2} = k^2e^{-\langle X, X \rangle }$ is constant. By symmetry, it suffices to look at trajectories with $\nu >0$. These trajectories all have $\tau$ and $\nu$ going to $\infty$ for a finite $s$. Similarly, when we follow the trajectories backwards, $\tau$ goes to $-\infty$ and $\nu$ to $\infty$ for a finite $s$. Each trajectory hits both lines $\tau=\nu$ and $\tau = -\nu$. Therefore, the corresponding curve $X$ is convex, crosses each light-like axis, has finite Minkowski-length with $k\rightarrow\infty$ on each end. 
\begin{figure}[t]
	\centering
	\begin{subfigure}[t]{0.48\textwidth}
		\centering
		\includegraphics{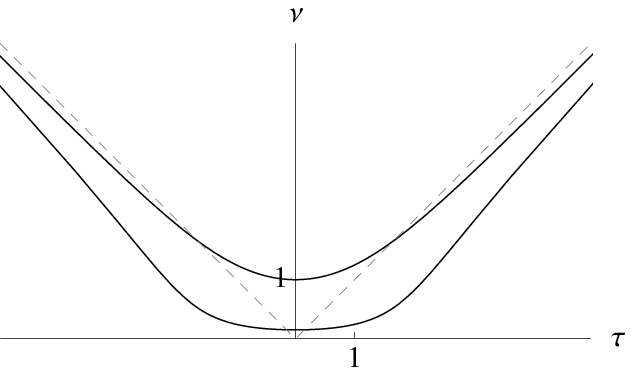}
		\caption{Phase portrait.}
		\label{contrasia}
	\end{subfigure}
	\quad
	\begin{subfigure}[t]{0.48\textwidth}
		\centering
		\includegraphics{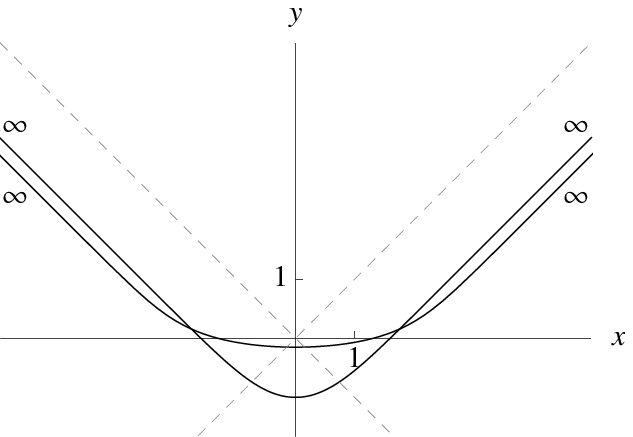}
		\caption{Two contracting curves.}
		\label{contrasib}
	\end{subfigure}
	\caption{Contraction, $a=0$, $b=-1$.}
	\label{contrasi}
\end{figure}

For a more detailed description, we consider the ODE for $y$ as a function of $x$,
\begin{equation*}
y''(x) = (1-y'(x)^2)(xy'(x)-y(x)).
\end{equation*}
The initial values $y(0) = -\alpha$, $y'(0) = 0$, correspond to the $(\tau, \nu)$-trajectory going through $(0,\alpha)$, so by varying $\alpha > 0$ we get all contracting curves. Also, with these initial values $y$ is an even function of $x$.

Now, it is clear that $y$ is convex, since $k>0$. Put $w(x) = xy'(x)-y(x)$. Since $w'(x) = xy''(x)$ and $y''(x) > 0$,  $w$ has a global minimum $w(0) = \alpha$. Therefore, $y''(x) \geq \alpha (1-y'(x)^2)$, so by integration we get $\text{artanh}\, y'(x) \geq \alpha x$, i.e., $y'(x) \geq \tanh \alpha x$, for $x>0$. Integrating again yields $y(x) \geq \frac 1 \alpha \log \cosh \alpha x - \alpha \geq x - \alpha - \frac{\log 2}{\alpha}$.  So the increasing function $x-y(x)$ has a limit in the interval $(\alpha, \alpha + \frac{\log 2}{\alpha})$.

Therefore, the curve $X$ is asymptotic to a cone. Under the flow, it comes from infinity as $t \rightarrow -\infty$, and contracts to the union of the positive $\xi$-axis and the negative $\eta$-axis, as $t \rightarrow \frac 1 2$. Its finite length decreases under the flow.
\end{proof}

In \cite{eck97}, Ecker proved that for every space-like curve, there exists a solution to the flow starting at the curve which exists for all $t\geq0$. For the curve $X$, this solution must be different from the contracting solution, which develops a singularity at $t=\frac 1 2$. This non-uniqueness should not be too surprising, as the curvature of $X$ blows up on each end.

\section{Hyperbolic rotation}
\label{rotating}

Now suppose $a\neq0$, $b=0$, so by scaling and reflecting the curve if necessary, we may assume $a=1$.  

\begin{thm}
There are three types of curves which move forever under a hyperbolic rotation with unit angular velocity.
\end{thm}

The curves appear in Figure \ref{rotib}.

\begin{figure}[t]
	\centering
	\begin{subfigure}[t]{0.48\textwidth}
		\centering
		\includegraphics{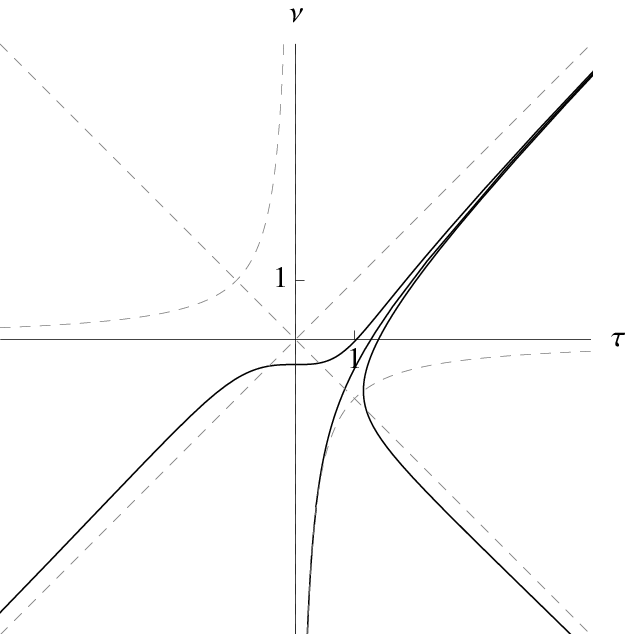}
		\caption{Phase portrait.}
		\label{rotia}
	\end{subfigure} 
	\quad
	\begin{subfigure}[t]{0.48\textwidth}
		\centering
		\includegraphics{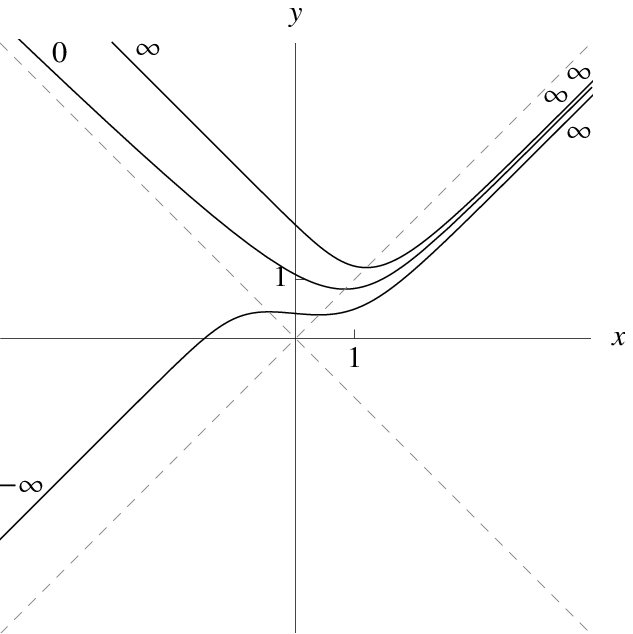}
		\caption{All three curve types.}
		\label{rotib}
	\end{subfigure}
	\caption{Hyperbolic rotation, $a=1$, $b=0$.}
	\label{roti}
\end{figure}

\begin{proof}
Here $k = \tau$ so the system of ODEs for $\tau$ and $\nu$ becomes
 \begin{equation*}
\label{Snuningshneppi}
\left\{
\begin{aligned}
\tau' &= 1 + \tau \nu \\
\nu'&= \tau^2.
\end{aligned}
\right.
\end{equation*}
The phase portrait appears in Figure \ref{rotia}.
The system has no fixed points and $\nu$ is increasing. Direct calculations show that the function $\tau^2-\nu^2-2\theta = \langle X, X \rangle - 2\theta$ is constant.
All trajectories in the right half of the phase plane have $\tau$ and $\nu$ going to $\infty$ for finite values of $s$. But when we follow a trajectory in the backwards direction, three things can happen:

The trajectory crosses the $\nu$-axis. Then, by symmetry, $\tau$ and $\nu$ go to $-\infty$ for a finite $s$. The corresponding curve $X$ has finite Minkowski-length with the curvature $k$ increasing from $-\infty$ to $\infty$, so the curve has an inflection point.

The trajectory crosses the curve $\tau \nu = -1$.  Then $\tau \rightarrow \infty$, $\nu \rightarrow -\infty$ for a finite $s$. The curve $X$ is convex, with finite Minkowski-length and $k\rightarrow \infty$ on each end.

The trajectory gets trapped between them the $\nu$-axis and $\tau\nu = -1$. Then $\tau \rightarrow 0$ and $\nu \rightarrow -\infty$, and this happens as $s\rightarrow -\infty$, since  $0 < \tau' < 1$. So the curve $X$ is convex, Minkowski-finite in one direction and Minkowski-infinite in the other, with $k$ increasing from $0$ to $\infty$.

To get a better description of the curves $X$, we look at the ODE for $\xi$ as a function of $\eta$,
\begin{equation*}
\xi''(\eta) = \xi'(\eta)(\eta\xi'(\eta)+\xi(\eta)).
\end{equation*}
Here the function $\xi'(\eta)e^{-\eta\xi(\eta)}$ is constant and by applying a hyperbolic rotation, we can assume the constant is $1$ (which corresponds to $\langle X, X \rangle = 2\theta$). So we are left with the ODE
\begin{equation*}
\xi'(\eta) = e^{\eta\xi(\eta)}.
\end{equation*}
It is easy to see that all solutions cross the $\xi$-axis. This corresponds to the fact that each $(\tau,\nu)$-trajectory crosses the line $\tau = -\nu$. By symmetry, it suffices to look at curves such that $\xi(0) \geq 0$. Each one of them blows up for a finite positive value of $\eta$. When we go in the negative $\eta$ direction, two things can happen:

The curve crosses the $\eta$ axis and goes to $-\infty$ for a finite negative value of $\eta$. It is then of the first kind in the trichotomy above. Under the flow, it goes from the $\eta$-axis, as $t\rightarrow -\infty$, to the $\xi$-axis, as $t \rightarrow \infty$.

The curve is bounded from below by a nonnegative number $L$. The curve is of the second kind when $L>0$, and of the third kind when $L=0$. Under the flow, it goes from the $\eta$-axis, as $t\rightarrow -\infty$, and out to infinity, as $t \rightarrow \infty$. 
\end{proof}

\section{Hyperbolic rotation and dilation}
\label{screwdilating}
Here we assume neither $a$ nor $b$ is $0$, so $X$ satisfies the full equation $k = a\tau - b\nu$. As we saw earlier, the screw-dilation factor is $((2bt+1)^{\frac{b+a}{2b}},(2bt+1)^{\frac{b-a}{2b}})$. The dilation is a contraction when $b<0$ and an expansion when $b>0$.

\begin{thm}
There are seven types of curves which move under a hyperbolic rotation and contraction, including the exact solution
\begin{itemize}
\item $\xi = (1-2t)\tanh \eta$, $t < \tfrac 1 2$, with $k = - \tan s$, $-\frac \pi 2 < s < \frac \pi 2$, at $t=0$.
\end{itemize}

There are eighteen types of curves which move under a hyperbolic rotation and expansion, including the exact solutions

 \begin{itemize}
 \item $\xi = (1+2t) \tan \eta$, $t > -\tfrac 1 2$, with $k = \tanh s$ at $t=0$,
\item $\xi = -(1+2t)\coth \eta$, $t > -\tfrac 1 2$, with $k = \coth s$, $s>0$, at $t=0$.
\end{itemize}
\end{thm}

\begin{figure}[t]
	\centering
	\begin{subfigure}[t]{0.48\textwidth}
		\centering
		\includegraphics{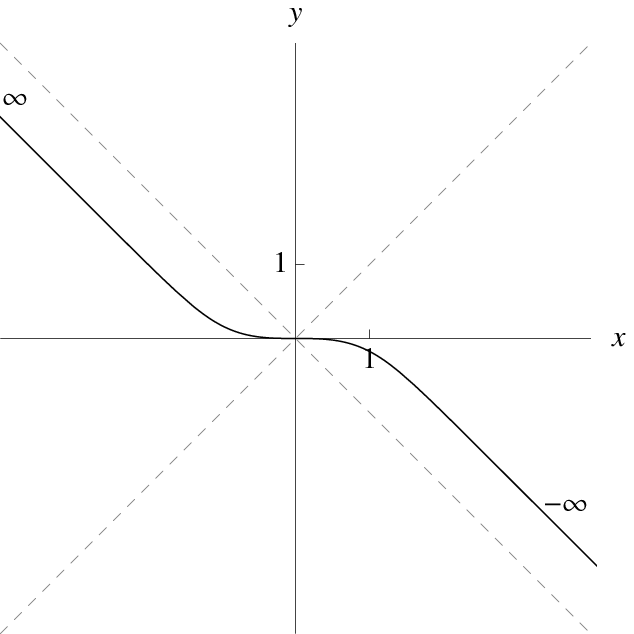}
		\caption{$\xi = (1-2t)\tanh \eta$.}
		\label{inn}
	\end{subfigure}
	\quad
	\begin{subfigure}[t]{0.48\textwidth}
		\centering
		\includegraphics{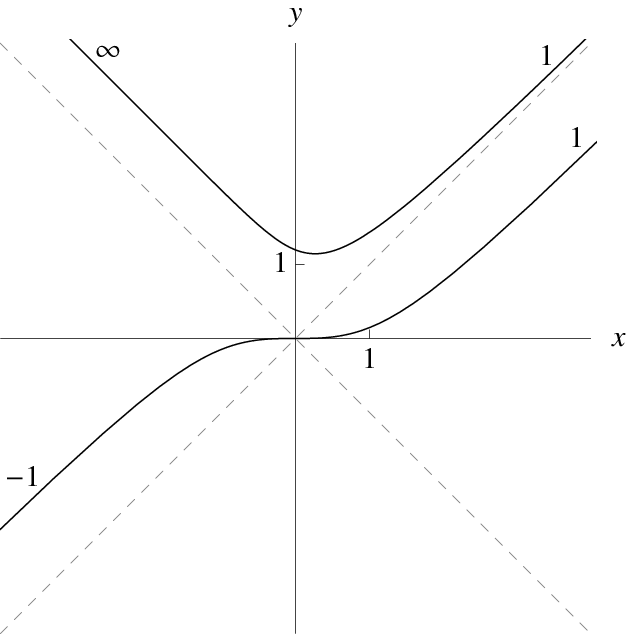}
		\caption{$ \xi = (1+2t) \tan \eta$, \quad $\xi = -(1+2t)\coth \eta$.}
		\label{ut}
	\end{subfigure}
	\caption{Hyperbolic rotation and dilation, $a=b=\mp1$.}
\end{figure}

The curves appear in Figures \ref{inn}, \ref{rotcontrasi1b}, \ref{rotcontrasi2b}, and Figures  \ref{ut}, \ref{rotexpansi1b},  \ref{rotexpansi1c},  \ref{rotexpansi2b}, \ref{rotexpansi2c}, respectively.

\begin{proof}

We start with the special case $a^2=b^2$, where the screw-dilation is just a linear dilation in one of the variables $\xi$ and $\eta$. By scaling the curve and possibly reflecting it across the $y$-axis, it suffices to consider $a=b=\pm 1$. The screw-dilation factor is $(2bt+1,1)$, so under the flow the curve simply scales linearly in the $\xi$-variable. The curves we find below are unique, up to a hyperbolic rotation and a translation in $\eta$.

Assume $a=b=-1$, i.e., dilation towards the $\eta$-axis. Then $k = -\tau+\nu$ and the ODE for $\xi$ as a function of $\eta$ is
\begin{equation*}
\xi''(\eta) = -2 \xi(\eta) \xi'(\eta).
\end{equation*}
This gives $\xi'(\eta) = -\xi(\eta)^2+A$, and since $\xi(\eta)$ is increasing, we must take $A>0$. By applying a hyperbolic rotation, we can assume $A=1$.
The solution is $\xi = \tanh \eta$, shown in Figure \ref{inn}. This curve has finite Minkowski-length $\pi$ (which decreases under the flow) and $k = - \tan s$, $-\frac \pi 2 < s < \frac \pi 2$ . The solution to the flow is
\begin{equation*}
\xi = (1-2t)\tanh \eta, \quad t < \tfrac 1 2.
\end{equation*}
It goes from the $\xi$-axis, as $t\rightarrow-\infty$, to the $\eta$-axis, as $t\rightarrow \frac 1 2$.

Assume $a=b=1$, i.e., dilation away from the $\eta$-axis. Then $k = \tau - \nu$ and the ODE for $\xi$ as a function of $\eta$ is
\begin{equation*}
\xi''(\eta) = 2 \xi(\eta) \xi'(\eta).
\end{equation*}
This gives $\xi'(\eta) = \xi(\eta)^2+A$. By applying a hyperbolic rotation, it suffices to look at $A=1$, $A=0$ and $A=-1$.

When $A=0$, we just get the space-like unit hyperbola $\xi = -\frac 1 \eta$, and the motion is the same as before (since the hyperbolic rotation is vacuous).

 When $A=1$, the solution is $\xi = \tan \eta$, shown in Figure \ref{ut}. It has infinite Minkowski-length and $k = \tanh s$. The solution to the flow is
 \begin{equation*}
 \xi = (1+2t) \tan \eta, \quad t > -\tfrac 1 2.
 \end{equation*}
 It goes from the broken line $\{\xi \leq 0, \eta = -\frac \pi 2\}\cup\{\xi = 0, -\frac \pi 2 \leq \eta \leq \frac \pi 2\}\cup\{\xi \geq 0, \eta = \frac \pi 2\}$, as $t\rightarrow -\frac 1 2$, to the $\xi$-axis, as $t \rightarrow \infty$.

When $A=-1$, the solution is $\xi = -\coth \eta$, $\eta < 0$, (or its reflection across the origin), shown in Figure \ref{ut}. This curve is Minkowski-finite in one direction and Minkowski-infinite in the other, with $k = \coth s$, $s>0$. The solution is
\begin{equation*}
\xi = -(1+2t)\coth \eta, \quad t > -\tfrac 1 2.
\end{equation*}
It goes from the broken line $\{\xi=0, \eta \leq 0\}\cup\{\xi \geq 0, \eta =0\}$, as $t\rightarrow -\frac 1 2$, and out to infinity, as $t \rightarrow \infty$.\\

\begin{figure}[t]
	\centering 
	\begin{subfigure}[t]{0.48\textwidth}
		\centering
		\includegraphics{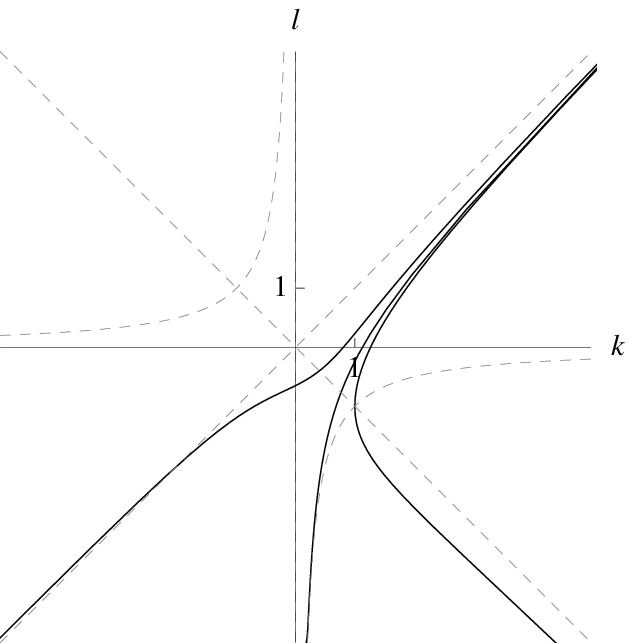}
		\caption{Phase portrait.}
		\label{rotcontrasi1a}
	\end{subfigure}
	\quad
	\begin{subfigure}[t]{0.48\textwidth}
		\centering
		\includegraphics{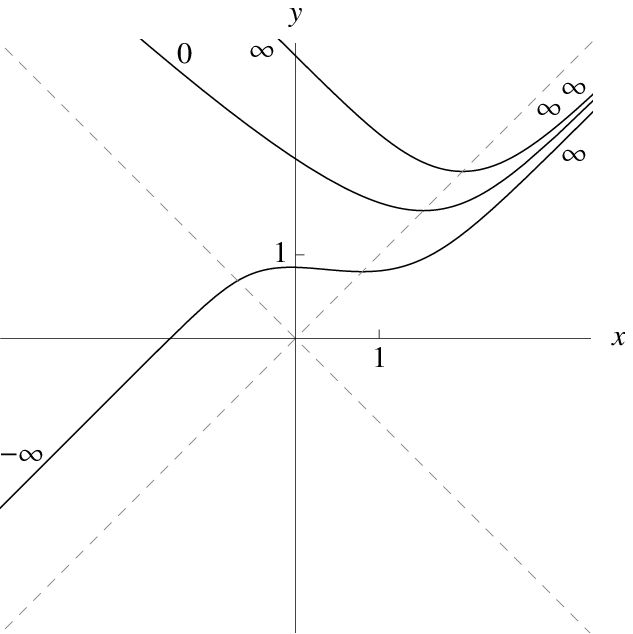}
		\caption{All three curve types.}
		\label{rotcontrasi1b}
	\end{subfigure}
	\caption{Hyperbolic rotation and contraction, $a=1$, $b=-\frac 1 2$.}
	\label{rotcontrasi1}
\end{figure}

Now consider the case where $a^2\neq b^2$.  Here we will for simplicity work with the system in $k$ and $l$, i.e., 
\begin{equation}
\label{klODE}
\left\{
\begin{aligned}
k' &= a + kl \\
l' &= -b + k^2,
\end{aligned} \right.
\end{equation}
where $k+hl=(a-hb)(\tau+h\nu)$.
Note that the ODE for $\xi$ as a function of $\eta$ can also be written as
\begin{equation}
\label{saur}
\frac{d}{d\eta}\left( \xi(\eta)\eta^{\frac{a+b}{a-b}} \right) = \frac{\eta^\frac{2b}{a-b}}{a-b}\frac{\xi''(\eta)}{\xi'(\eta)}.
\end{equation}
We already saw a special case of this for rotating curves ($b=0$).

By scaling and reflecting the curve, we may assume $a=1$. We look separately at the cases where $b<0$ and $b>0$.

\begin{figure}[t]
	\centering
	\begin{subfigure}[t]{0.48\textwidth}
		\centering
		\includegraphics{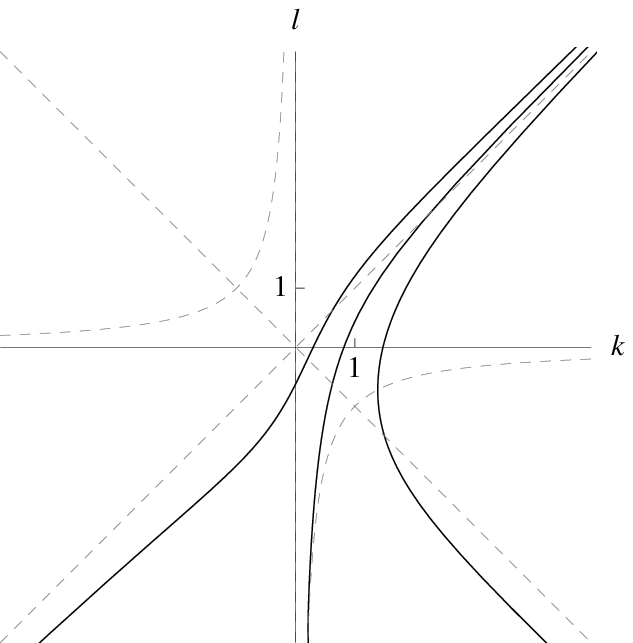}
		\caption{Phase portrait.}
		\label{rotcontrasi2a}
	\end{subfigure}
	\quad
	\begin{subfigure}[t]{0.48\textwidth}
		\centering
		\includegraphics{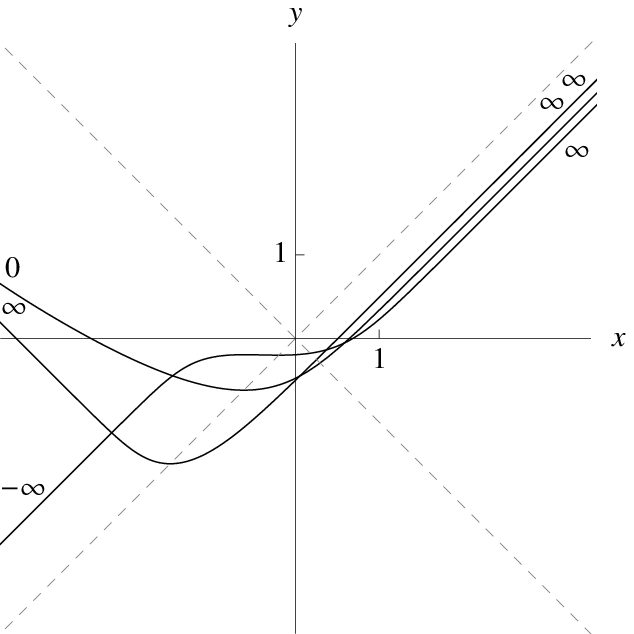}
		\caption{All three curve types.}
		\label{rotcontrasi2b}
	\end{subfigure}
	\caption{Hyperbolic rotation and contraction, $a=1$, $b=- 2$.}
	\label{rotcontrasi2}
\end{figure}
First assume $b<0$ and write $b = -\beta^2$ where $\beta>0$. The system \eqref{klODE} then takes the form
\begin{equation*}
\left\{
\begin{aligned}
k' &= 1 + kl \\
l' &= \beta^2 + k^2.
\end{aligned} \right.
\end{equation*}
The phase portrait appears in Figures \ref{rotcontrasi1a}, \ref{rotcontrasi2a}.
This system behaves similarly to the $(\tau,\nu)$-system for the rotating curves, which is of course nothing but the special case  $\beta=0$. 
There are no fixed points and $l$ is strictly increasing. Each trajectory in the right half-plane ($k>0$) has $k$ and $l$ going to $\infty$ for a finite value of $s$. When we follow a trajectory in the backwards direction, three things can happen:

\begin{figure}[t]
	\centering
	\begin{subfigure}{0.8\textwidth}
		\centering
		\includegraphics{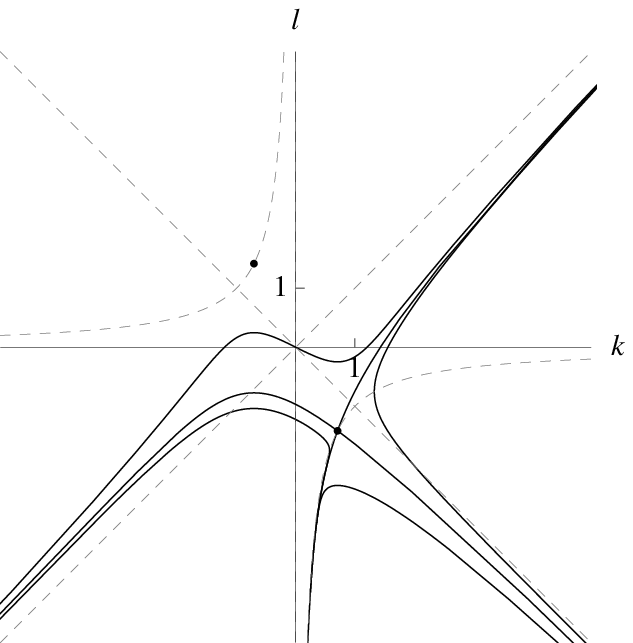}
		\caption{Phase portrait.}
		\label{adam1}
		\vspace{0.1in}
	\end{subfigure}
	\begin{subfigure}[t]{0.48\textwidth}
		\centering
		\includegraphics{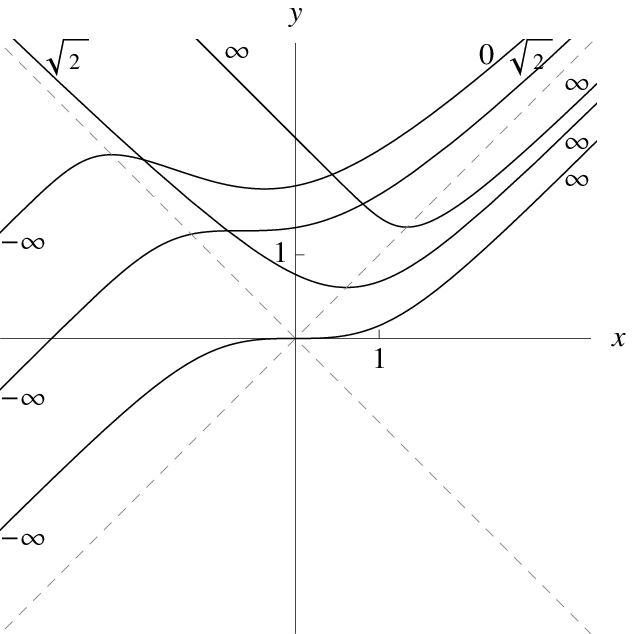}
		\caption{Five curve types.}
		\label{rotexpansi1b}
	\end{subfigure} 
	\quad
	\begin{subfigure}[t]{0.48\textwidth}
		\centering
		\includegraphics{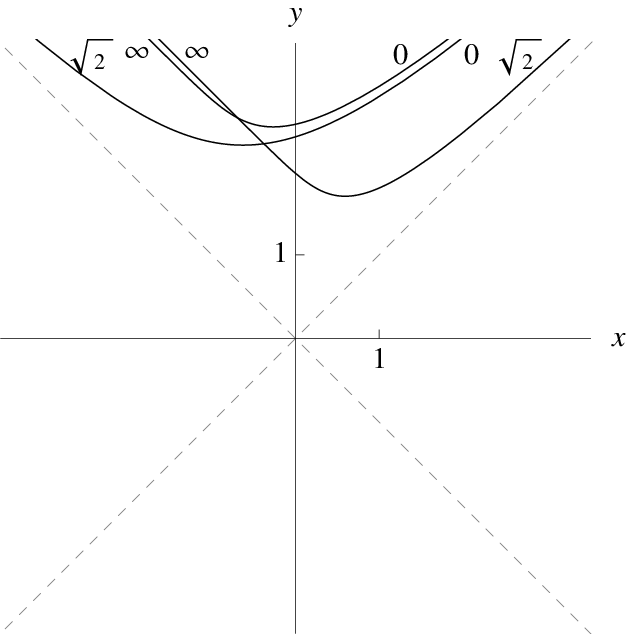}
		\caption{Three curve types.}
		\label{rotexpansi1c}
	\end{subfigure}
	\caption{Hyperbolic rotation and expansion, $a=1$, $b=\tfrac 1 2$.}
	\label{rotexpansi1}
\end{figure}

The trajectory crosses the $l$-axis. Then, by symmetry, $k$ and $l$ go to $-\infty$ for a finite $s$. The corresponding curve $X$ has finite Minkowski-length (which decreases under the flow) with the curvature $k$ increasing from $-\infty$ to $\infty$, so the curve has an inflection point.

The trajectory crosses the curve $kl = -1$.  Then $k \rightarrow \infty$, $l \rightarrow -\infty$ for a finite $s$. The curve $X$ is convex with finite Minkowski-length and $k\rightarrow\infty$ on each end.

The trajectory gets trapped between the $l$-axis and $kl = -1$. Then $k \rightarrow 0$ and $l \rightarrow -\infty$, and this happens as $s\rightarrow -\infty$, since  $0 < k' < 1$. So the curve $X$ is convex, Minkowski-finite in one direction and Minkowski-infinite in the other, with $k$ increasing from $0$ to $\infty$.

\begin{figure}[t]
	\centering
	\begin{subfigure}[t]{0.8\textwidth}
		\centering
		\includegraphics{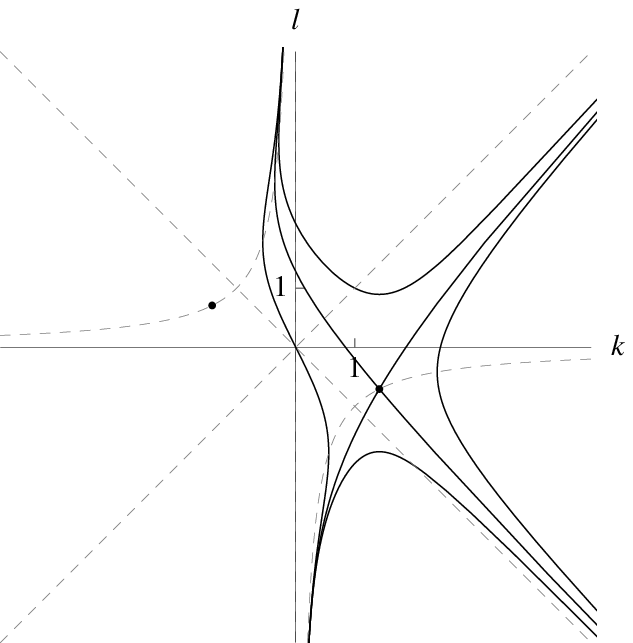}
		\caption{Phase portrait.}
		\label{adam2}
		\vspace{0.1in}
	\end{subfigure}
	\begin{subfigure}[t]{0.48\textwidth}
		\centering
		\includegraphics{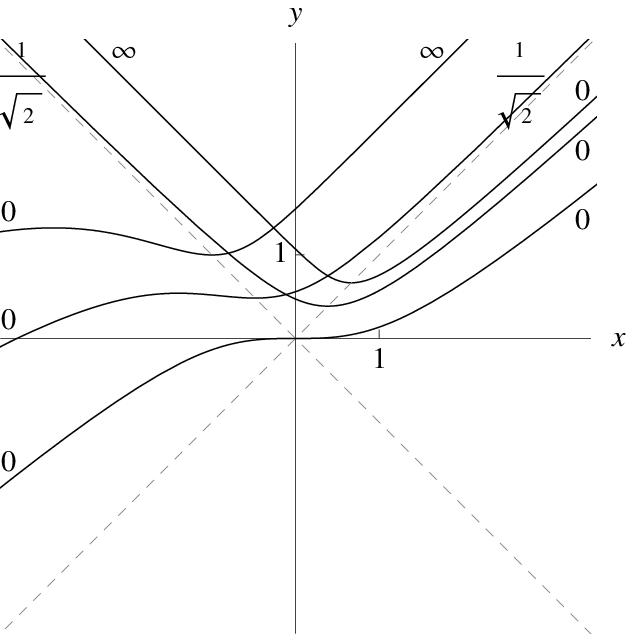}
		\caption{Five curve types.}
		\label{rotexpansi2b}
	\end{subfigure}
	\quad
	\begin{subfigure}[t]{0.48\textwidth}
		\centering
		\includegraphics{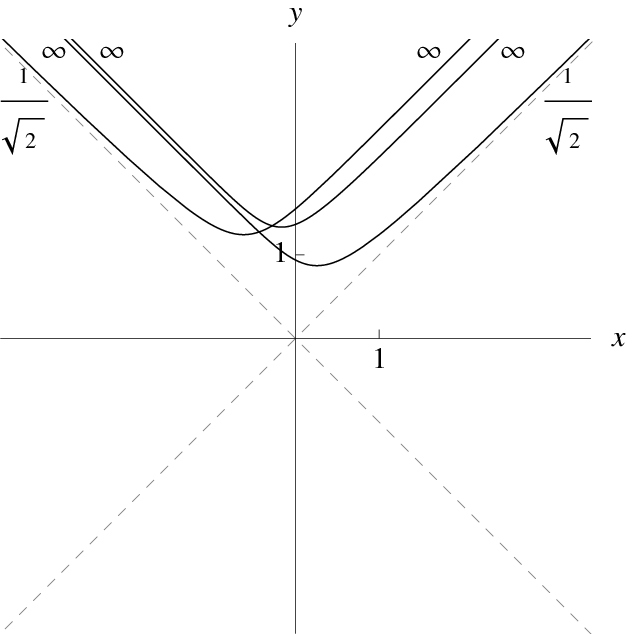}
		\caption{Three curve types.}
		\label{rotexpansi2c}
	\end{subfigure}
	\caption{Hyperbolic rotation and expansion, $a=1$, $b=2$.}
	\label{rotexpansi2}
\end{figure}

As in the case of rotating curves, each of these trajectories crosses the line $k = -l$, which corresponds to $X$ crossing the $\xi$-axis. When $\beta >1$, each trajectory also crosses the line $k=l$, corresponding to $X$ crossing the $\eta$-axis. However, when $\beta < 1$, this only happens to the trajectories of the first kind above, just like in the case of rotating curves. This difference should not be too surprising, considering that the screw-dilation factor has the form $\left((1-2\beta^2t)^\frac{\beta^2-1}{2\beta^2},(1-2\beta^2t)^\frac{\beta^2+1}{2\beta^2}\right)$, so the behaviour of the $\xi$-component is quite different for $\beta<1$ and $\beta > 1$. The curves can be seen in Figures \ref{rotcontrasi1b}, \ref{rotcontrasi2b}.

Now, assume $b>0$ and write $b = \beta^2$ where $\beta > 0$. Then the system takes the form
\begin{equation*}
\left\{
\begin{aligned}
k' &= 1 + kl \\
l' &=  k^2 - \beta^2.
\end{aligned} \right.
\end{equation*}
In the right half of the phase plane, we have a fixed point $(\frac 1 \beta, -\beta)$, which corresponds to $X$ being the expanding hyperbola, making the rotation vacuous. The fixed point is a saddle point. There are two trajectories coming into it, two coming out of it, and between these four trajectories we have four families of trajectories, as can be seen in the phase portraits in Figures \ref{adam1}, \ref{adam2}. So there are eight types of curves (excluding the hyperbola). As before, the curves show different behaviour depending on whether $\beta <1$ or $\beta > 1$. Instead of describing each of these sixteen curve types in detail, we refer to Figures \ref{rotexpansi1b}, \ref{rotexpansi1c}, \ref{rotexpansi2b}, \ref{rotexpansi2c}.  We just mention that by \eqref{saur}, when $\beta >1$, the function $\xi(\eta)\eta^{-\frac{\beta^2+1}{\beta^2-1}}$ is decreasing for $\eta >0$, preventing $\xi$ from blowing up for a finite positive $\eta$.
\end{proof}

\section{Hyperbolic rotation, dilation and translation}
\label{screwdilatingtranslating}

The only self-similar motion under the flow that remains to be investigated is the one with $a=b\neq0$ and $C = (0,c)$, i.e., dilation with factor $(2bt+1)$ in the $\xi$-variable and translation of $\frac{c}{2b}\log(2bt+1)$ in the $\eta$-variable. By scaling the curve if necessary, we may assume $|b|=1$. By applying a hyperbolic rotation to $X$ and reflecting it across the origin if necessary, we can assume $C=(0,1)$.

\begin{thm}
There are eight types of curves that move under a combination of hyperbolic rotation, dilation and translation. Five of them are expanding and three contracting, given respectively by the equations
\begin{equation*}
\eta = \pm \frac 1 2 \int\frac{e^\xi}{e^{\xi}\xi-e^\xi+A}d\xi,
\end{equation*}
for different values of the constant $A$.
\end{thm}

These curves appear in Figures  \ref{allir1} and \ref{allir2}.

\begin{proof}

 Instead of looking separately at the cases $b=1$ and $b=-1$, we will only consider $b=1$ and find both the space-like and time-like curves that satisfy the corresponding equation. Reflecting the time-like curves across the $\xi$-axis then gives the space-like curves with $a=b=-1$ and $C = (0,1)$.

Now, the ODE for $\xi$ as a function of $\eta$ is
\begin{equation*}
\xi''(\eta) = \xi'(\eta)(2\xi(\eta)-\xi'(\eta)).
\end{equation*}
Direct calculations show that the function $e^{\xi}(\tfrac{1}{2}\xi'-\xi+1)$ is constant. If we call the constant $A$, we have
\begin{equation}
\label{klobbi}
\xi' = 2\xi-2+2Ae^{-\xi}.
\end{equation}
The phase portrait of $(\xi,\xi')$ can be seen in Figure \ref{xiximerkt}. By separation of variables, the curves are given by the integral
\begin{equation}
\label{krunk}
\eta = \frac 1 2 \int\frac{e^\xi}{e^{\xi}\xi-e^\xi+A}d\xi.
\end{equation}
We divide the values of $A$ into a few cases, depending on the number of zeros in the denominator.
 \begin{figure}[t]
 	\centering
	\begin{subfigure}[t]{0.8\textwidth}
		\centering
		\includegraphics{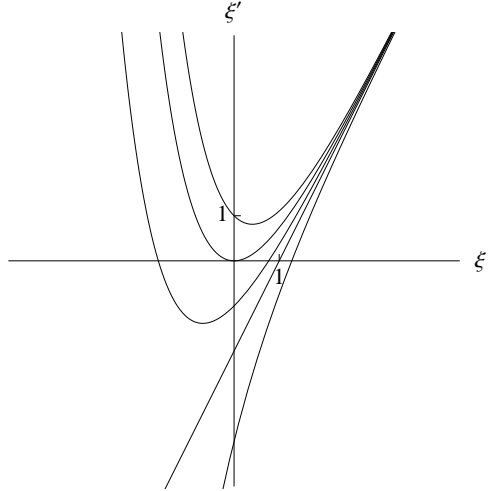}
		\caption{Phase portrait of $(\xi,\xi')$, showing the trajectories for $A=\frac 3 2$, $A=1$, $A=\frac 1 2$, $A=0$ and $A=-1$ (in order of decreasing $\xi'$-intercept).}
		\label{xiximerkt}
			\vspace{0.1in}
	\end{subfigure}
	\begin{subfigure}[t]{0.48\textwidth}
		\centering
		\includegraphics{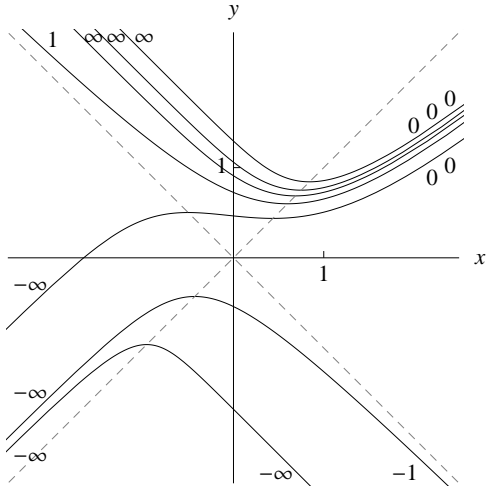}
		\caption{In order of decreasing $y$-intercept,\\ $A=-1, 0, \frac 1 2, 1, \frac 3 2, 1, \frac 1 2$.}
		\label{allir1}
	\end{subfigure}
	\quad
	\begin{subfigure}[t]{0.48\textwidth}
		\centering
		\includegraphics{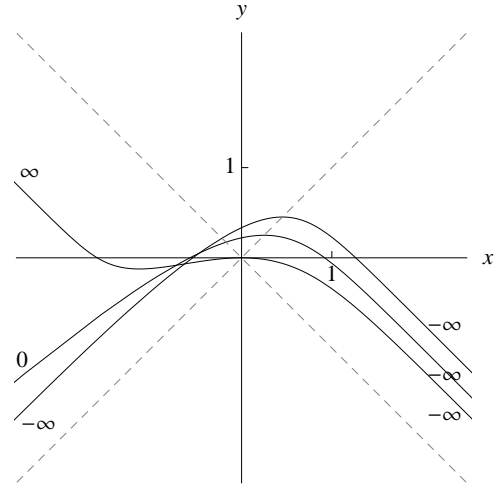}
		\caption{In order of decreasing $y$-intercept,\\ $A=-1, 0, \frac1 2$.}
		\label{allir2}
	\end{subfigure}
	\caption{Hyperbolic rotation, dilation and translation, $C=h_-$ and $a=b=\pm1$ on the left and right, respectively.}
\end{figure}

\begin{itemize}
\item When $A\leq0$, the denominator has a single root. We get a convex space-like curve and a convex time-like curve. 

\item When $0<A<1$, the denominator has two single roots. We get two convex space-like curves and a time-like curve with an inflection point.

\item When $A=1$, the denominator has a double root. We get two convex space-like curves.

\item When $A>1$, the denominator has no roots. We get a space-like curve with an inflection point.
\end{itemize}
The curvature is given by
\begin{equation*}
k = \frac{1-Ae^{-\xi}}{(2\xi-2+2Ae^{-\xi})^{\frac 1 2}}
\end{equation*}
which follows from \eqref{xijofnur} and \eqref{klobbi}. The three curves at the top Figure \ref{allir1} have the same curvature limits so we consider them as being of the same type.
\end{proof}

In the special case $A=0$, the curves are the space-like $\xi = e^{2\eta}+1$ and the time-like $\xi = -e^{2\eta}+1$, which after reflection becomes the space-like $\xi = -e^{-2\eta}+1$. Note that in Section \ref{translating} we saw that the curve $\xi = e^{2\eta}+1$ translates with constant velocity vector $2h_+$ under the flow. For this curve, the two self-similar motions turn out to be the same (up to reparametrization), although one of them is defined only for $t > -\frac 1 2$ and the other one for all $t \in \mathbf R$. Similarly, the self-similar motion of $\xi = -e^{-2\eta}+1$ is identical (up to reparametrization) to the translation with constant velocity vector $-2h_+$, even though the former is only defined for $t < \frac 1 2$. The reason for this is explained in the Appendix.

The space-like solution when $A=0$ comes out of the smooth curve $\xi = e^{2\eta}$ at time $t=-\frac 1 2$. Something similar holds for other values of $A$. Let $\lambda_A$ be the biggest root in the denominator of \eqref{krunk} (or $\lambda_A = 0$ if $A > 1$) and take any $\mu_A > \lambda_A$. Then we have a space-like solution which at time $t$ is given by
\begin{equation*}
\eta = \tfrac 1 2 \log (2t+1) +\tfrac 1 2 \int_{\mu_a}^{\frac{\xi}{2t+1}}\frac{e^u}{e^u u-e^u+A}du, \quad \xi > (2t+1)\lambda_A.
\end{equation*}
By direct calculations, it can be shown that for each $\xi > 0$, the right hand side converges to $\frac 1 2 \log \xi + C(A, \mu_A)$,  as $t \rightarrow -\frac 1 2$, where $C$ is a constant that we can get rid of by translating the curve in the $\eta$-variable. Hence these solutions also come out of the curve $\xi = e^{2 \eta}$ at time $ t = -\frac 1 2$. So for each $A$, there is a different self-similar solution to the flow coming out of the space-like curve $\xi = e^{2\eta}$, the one with $A=0$ coinciding with the translating solution. This non-uniqueness of the flow should not be too surprising, as this curve is Minkowski-finite in one direction where its curvature blows up.

\section{Other exact solutions}
\label{other}

We conclude this paper with the derivation of five simple exact solutions to the flow that are not self-similar.

Recall that by Lemma \ref{transformers}, if $y(x,t)$ is a solution in the Euclidean plane which is analytic and even in $x$, then $\hat y(x,t) = y(ix,-t)$ is a solution in the Minkowski plane, where $i^2=-1.$ We have already seen examples of this. The contracting circle in the Euclidean plane,
\begin{equation*}
x^2 + y^2 = -2t, \quad t <0,
\end{equation*}
transforms into the expanding space-like hyperbola in the Minkowski plane,
\begin{equation*}
-x^2  + y^2 = 2t, \quad t > 0.
\end{equation*}
The downward translating Grim Reaper in the Euclidean plane,
\begin{equation*}
\cos x = e^{y+t}, \quad t \in \mathbf R,
\end{equation*}
becomes the space-like upward translating solution in the Minkowski plane, that we found in Section \ref{translating},
\begin{equation*}
\cosh x = e^{y-t}, \quad t \in \mathbf R.
\end{equation*}

The same method works if the Euclidean solution is analytic and odd in $x$ if we are able to cancel out the extra $i$ factor that we get. In the presence of a factor of $e^t$, this can be done by a complex translation in $t$, since $e^{i\frac \pi 2}=i$. As an example, note that by translating the Grim Reaper in the $x$ variable, we can write it as
\begin{equation*}
\sin x = e^{y+t}, \quad t \in \mathbf R.
\end{equation*}
If we replace $x$ by $ix$ and $t$ by $-t + i \pi/2$ in this equation, we get the time-like upward translating solution in the Minkowski plane from Section \ref{translating},
\begin{equation*}
\sinh x = e^{y-t}, \quad t \in \mathbf R.
\end{equation*}

Now, in the Euclidean plane, two examples are known of exact solutions to the flow which are not self-similar. They are Angenent's oval/paperclip \cite{ang92},
\begin{equation*}
\cos x = e^t \cosh y, \quad t < 0,
\end{equation*}
and its sister solution, the Grim Reaper wave  \cite{broadvass11,king00},
\begin{equation*}
\cos x = e^t \sinh y, \quad t\in \mathbf R,
\end{equation*}
which can be seen in Figure \ref{kleina}.

\begin{figure}[t]
	\centering
	\begin{subfigure}[t]{0.34\textwidth}
		\centering
		\includegraphics{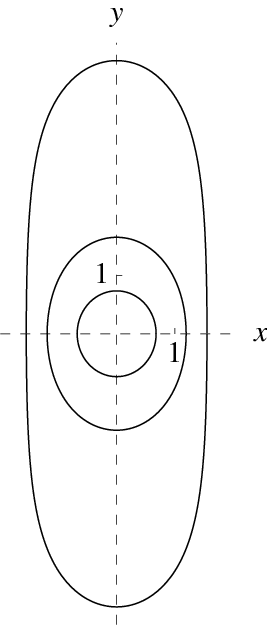}
		\caption{Angenent's oval, $\cos x = e^t \cosh y$, at $t=-4, -1,-\frac 1 4$.}
	\end{subfigure} 
	\quad\quad
	\begin{subfigure}[t]{0.56\textwidth}
		\centering
		\includegraphics{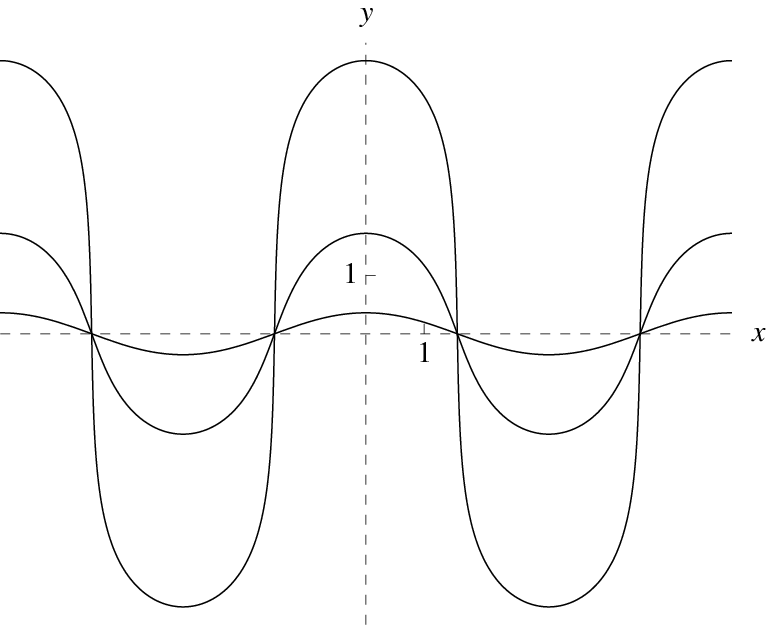}
		\caption{The Grim Reaper wave, $\cos x = e^t \sinh y$, at $t=-4, -1, 1$.}
	\end{subfigure}
	\caption{Exact non-self-similar solutions in the Euclidean plane.}
	\label{kleina}
\end{figure}

By applying the transformations to these solutions (i.e., replacing $\cos x$ with either $\cosh  x$ or $\sinh x$ and swapping the sign of $t$), we get three exact space-like solutions in the Minkowski plane. The fourth one is obtained by first reflecting Angenent's oval across $y=x$, then replacing $\cosh x$ with $\cos x$ and swapping the sign of $t$, and finally translating it by $\frac \pi 2$ along each axis (for aesthetic reasons).

\begin{thm}
The following are space-like solutions to the MCF in the Minkowski plane:
\begin{itemize}
\item $\cosh x = e^{-t} \cosh y, \quad t >0$,
\item $\cosh x = e^{-t} \sinh y, \quad t \in \mathbf R$,
\item $\sinh x = e^{-t} \sinh y, \quad t<0$,
\item $\sin y = e^{-t} \sin x, \quad t >0$.
\end{itemize}
\end{thm}
The curves appear in Figure \ref{flosi}.
Each of the first three solutions behaves like a self-similar solution near each of the two boundaries of the time interval where it is defined, serving as some sort of interpolation between them. This can be made precise by comparing the Taylor approximations and looking at Figure \ref{lengdarhasar}, where the Minkowski-lengths of the curves are shown as a function of $t$.

\begin{figure}[t]
	\centering
	\begin{subfigure}[t]{0.48\textwidth}
		\centering
		\includegraphics{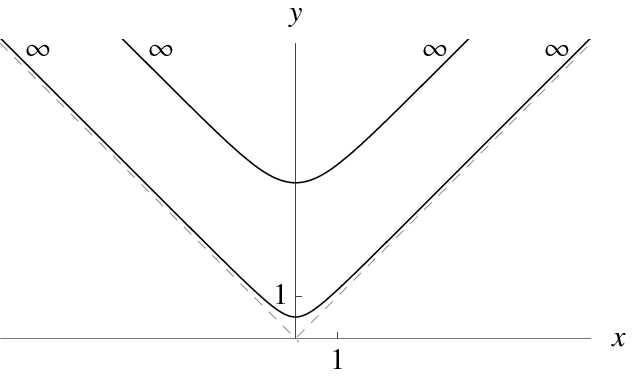}
		\caption{$\cosh x = e^{-t} \cosh y$ for $t=\frac 1 8, 3$.}
		\label{flosi1}
	\end{subfigure}
		\vspace{0.1in}
	\quad
	\begin{subfigure}[t]{0.48\textwidth}
		\centering
		\includegraphics{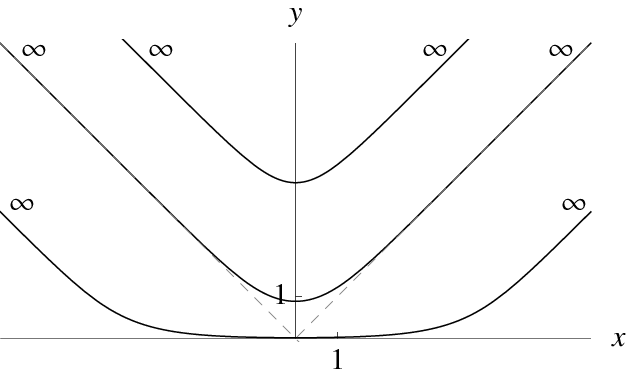}
		\caption{$\cosh x = e^{-t} \sinh y$ for $t=-4,0,3$.} 
		\label{flosi2}
	\end{subfigure}
	\begin{subfigure}[t]{0.48\textwidth}
		\centering
		\includegraphics{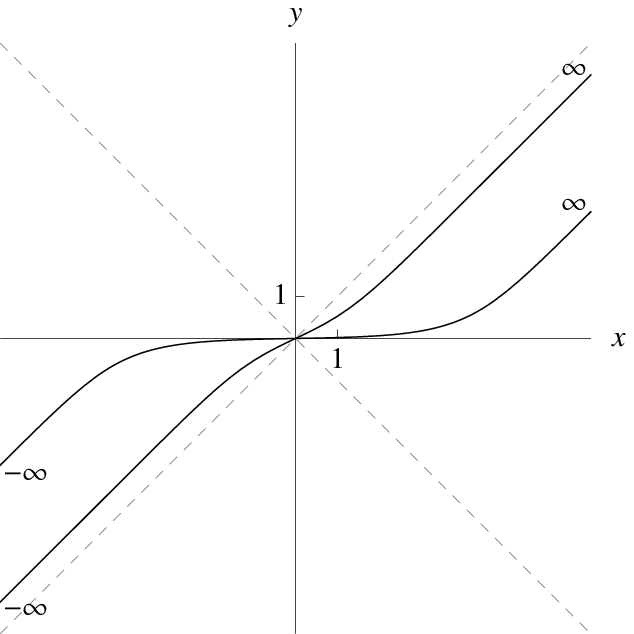}
		\caption{$\sinh x = e^{-t} \sinh y$ for $t=-4,-\frac 3 4$.}
		\label{flosi3}
	\end{subfigure}
	\quad
	\begin{subfigure}[t]{0.48\textwidth}
		\centering
		\includegraphics{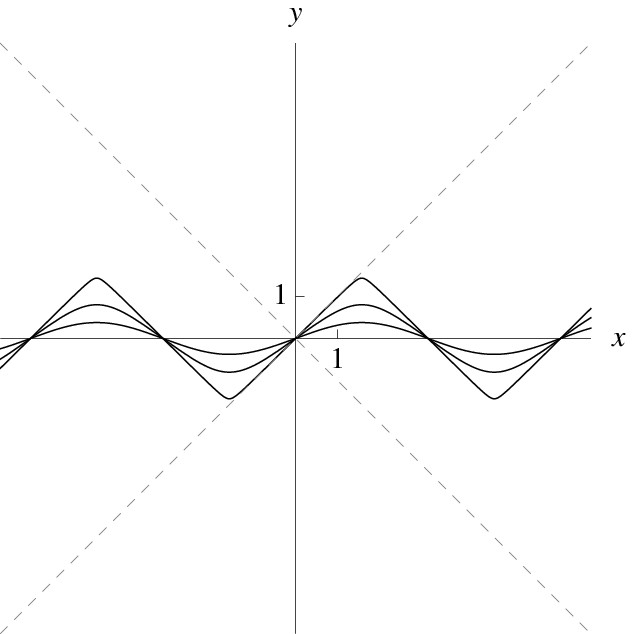}
		\caption{$\sin y = e^{-t} \sin x$ for $t=\frac{1}{100},\frac 1 3, 1$.}
	\end{subfigure}
	\caption{Exact non-self-similar solutions in the Minkowski plane.}
\label{flosi}
\end{figure}

The first solution consists of a curve in the upper half-plane (and its reflection across the $x$-axis) and it behaves like the expanding hyperbola $y^2=x^2+2t$ for $t$ close to $0$, and like the upward translating solution $\cosh x = e^{-t+y}$ for $t$ close to $\infty$. The Minkowski-length of the curve is finite and increases with $t$ from $0$ to $\pi$.

The second solution behaves like the right and leftward translating solutions $e^{\mp x} = e^{-t}\sinh y$ for $t$ close to $-\infty$, and like the upward translating solution $\cosh x = e^{-t+y}$ for $t$ close to $\infty$. The Minkowski-length of the curve is finite and decreases with $t$ down to the limit $\pi$.

The third solution behaves like the right and leftward translating solutions $\mp e^{\mp x} = e^{-t}\sinh y$ for $t$ close to $-\infty$, and converges to the $\xi$-axis as $t\rightarrow 0$, in a manner similar to the screw-contracting solution $x-y = -t\tanh \frac{x+y}{2}$.
The Minkowski-length of the curve is finite and decreases with $t$ down to $0$.

The fourth solution is in some sense the Minkowski analogue of the Euclidean Grim Reaper wave. It is periodic in $x$ with period $2\pi$ and comes out of the triangle wave at $t=0$, leaving each corner like the expanding hyperbola. As $t \rightarrow \infty$, it converges to the $x$-axis.\\

Another way to find exact solutions is to work with the PDE \eqref{kPDE2}, i.e., $k_t = k^2k_{\theta\theta} - k^3$. Note that if we impose the condition $k_\theta = 0$ or $k_t=0$, we get four familiar exact solutions:
\begin{itemize}
\item $k = \frac{1}{\sqrt{2t}}$, $t>0$: $X$ is the expanding hyperbola $y^2-x^2=2t$, $t>0$.
\item $k = \cosh \theta$: $X$ is the translating solution $\cosh x = e^{y-t}$, $t\in \mathbf R$.
\item $k = -\sinh \theta$, $\theta < 0$: $X$ is the translating solution $\sinh y = e^{t-x}$, $t\in \mathbf R$.
\item $k = e^{-\theta}$: $X$ is the translating solution $\xi = e^{\eta}+t$, $t\in \mathbf R$.
\end{itemize}

\begin{figure}[t]
	\centering
	\includegraphics{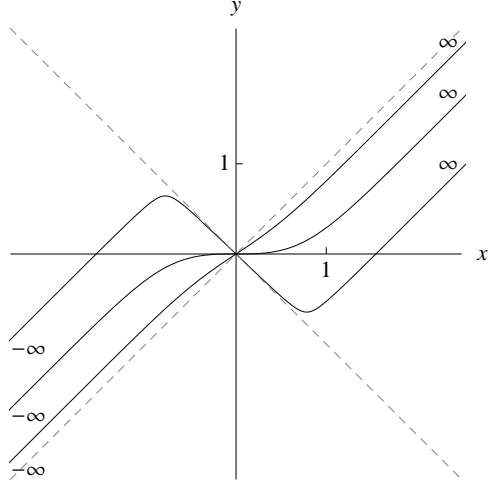}
	\caption{The exact non-self-similar solution $\tanh \xi = \tan \eta \tan 2t$ for $t=\frac{1}{100}, \frac{\pi}{8},\frac{\pi}{4}-\frac{1}{10}$.}
	\label{glaenyr}
\end{figure}

\begin{figure}[t]
	\centering
	\includegraphics{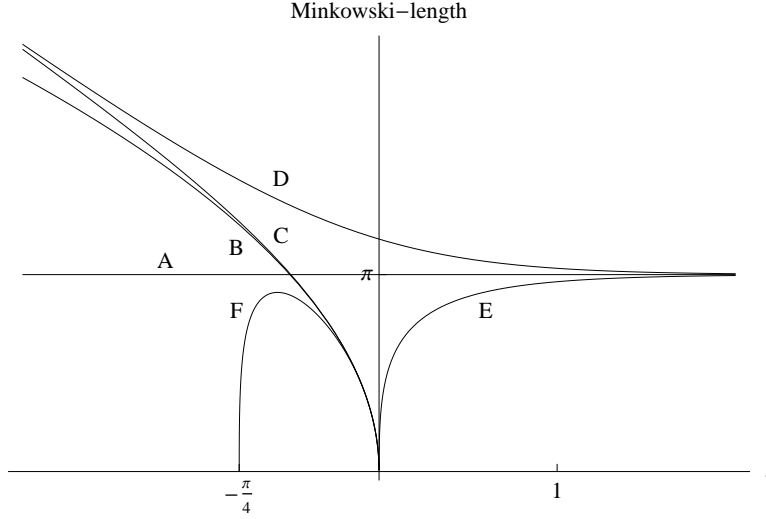}
	\caption{The Minkowski-lengths of all the finite length exact solutions as a function of $t$.  \textbf A: $\cosh x = e^{y-t}$. \textbf B: $\xi = -2t \tanh \eta$. \textbf C: $\sinh x = e^{-t} \sinh y$. \textbf D: $\cosh x = e^{-t} \sinh y$. \textbf E: $\cosh x = e^{-t} \cosh y$. \textbf F: $\tanh \xi = \tan \eta \tan (2t + \pi/2)$.}
	\label{lengdarhasar}
\end{figure}

Now consider solutions of the form $k(\theta,t)^2 = A(\theta) + B(t)$. With this substitution, the PDE takes the form
\begin{equation*}
A''(\theta)A(\theta) - \tfrac 1 2 A'(\theta)^2-2A^2(\theta) +\left(A''(\theta)-4A(\theta)\right)B(t) - 2B(t)^2-B'(t)=0.
\end{equation*}
In order for this to hold, we need $A''(\theta)-4A(\theta)$ to be constant, so (after possibly rescaling, reflecting and rotating the curve, and letting any additive constant be included in $B$) we have $5$ possibilities: $A(\theta)=\pm e^{-2\theta}$, $A(\theta) = \pm \cosh 2 \theta$, $A(\theta) = \sinh 2\theta$. Then $A''(\theta)-4A(\theta) = 0$ and $A''(\theta)A(\theta) - \tfrac 1 2 A'(\theta)^2-2A^2(\theta)$ takes on the constant values  $0,0,2,2,-2$, respectively. If we denote the constant by $2C$, we are left with the easily solvable ODE $B'(t) = 2(C-B(t)^2)$. Thus we get the following exact solutions, where in each case, $X$ is found by integration:
\begin{itemize}
\item $k^2 = e^{-2\theta} + \frac{1}{2t}$, $t<0$: $X$ is the solution $\xi = -2t \tanh \eta$, $t<0$.
\item $k^2 = e^{-2\theta} + \frac{1}{2t}$, $t>0$: $X$ is the solution $\xi = -2t \coth \eta$, $t>0$.
\item $k^2 = -e^{-2\theta} + \frac{1}{2t}$, $t >0$: $X$ is the solution $\xi = 2t \tan \eta$, $t >0$.
\item $  k^2 = \cosh 2\theta + \coth 2t$, $t<0$: $X$ is (after rescaling) the solution $\sinh x = e^{-t} \sinh y$, $t<0$.
\item $  k^2 = \cosh 2\theta + \coth 2t$, $t>0$: $X$ is (after rescaling) the solution  $\cosh x = e^{-t} \cosh y$, $t >0$.
\item $  k^2 = \cosh 2\theta + \tanh 2t$,  $t\in \mathbf R$: $X$ is (after rescaling) the solution $\cosh x = e^{-t} \sinh y$, $t \in \mathbf R$.
\item $  k^2 = -\cosh 2\theta + \coth 2t$, $t>0$: $X$ is (after rescaling) the solution $\sin y = e^{-t} \sin x$, $t>0$.
\item $  k^2 = \sinh 2\theta + \cot 2t$, $0<t<\frac{\pi}{2}$: $X$ is (after rescaling) the solution $\tanh \xi = \tan \eta \tan 2t$, $0<t<\frac{\pi}{4}$.
\end{itemize}

We recover all our previous exact solutions in addition to one new solution:

\begin{thm}
The following is a solution to the MCF in the Minkowski plane
\begin{equation*}
\tanh \xi = \tan \eta \tan 2t, \quad 0<t<\tfrac{\pi}{4}.
\end{equation*}
\end{thm}
The curve appears in Figure \ref{glaenyr}. As $t \rightarrow 0$, it comes out of  the broken line $\{\xi \leq 0, \eta = -\frac \pi 2\}\cup\{\xi = 0, -\frac \pi 2 \leq \eta \leq \frac \pi 2\}\cup\{\xi \geq 0, \eta = \frac \pi 2\}$, behaving like the solution $\xi = 2t \tan \eta$ as $t \rightarrow 0+$. As $t \rightarrow \frac \pi 4$, it converges to the $\xi$-axis, behaving like the solution $\eta = -2t \tanh \xi$ as $t \rightarrow 0-$. So this new solution serves as some sort of interpolation between these two exact self-similar solutions. Moreover, it's Minkowski-length is finite, increases from $0$ to a maximum value and then decreases to $0$ again.

Note that we have a total of twelve exact solutions in the Minkowski plane, compared to only four in the Euclidean plane (known to the author). The curvature of the Euclidean solutions is as follows:
\begin{itemize}
\item $k = \frac{1}{\sqrt{-2t}}$, $t<0$: $X$ is the contracting circle $x^2+y^2=-2t$, $t<0$.
\item $k = \cos \theta$: $X$ is the translating Grim Reaper $\cos x = e^{t-y}$, $t\in \mathbf R$.
\item $ k^2 = \cos 2\theta - \coth 2t$, $t<0$: $X$ is (after rescaling) Angenent's oval $\cos x = e^t \cosh y$, $t < 0$.
\item $k^2 = \cos 2\theta - \tanh 2t $, $t\in \mathbf R$: $X$ is (after rescaling) the Grim Reaper wave $\cos x = e^t \sinh y$, $t\in \mathbf R$.
\end{itemize}
In particular, we see that many of the exact Minkowski solutions above correspond to these four exact Euclidean solutions under the transformation from Lemma \ref{transformers2}, which would have been another way of deriving them.

\section*{Appendix}
\label{appendix}

In this Appendix we find all space-like curves in the Minkowski plane which are invariant under some self-similar motion. 

As a warm up, we solve the same task for curves in the Euclidean plane. Let $X:I \rightarrow \mathbf R^{2}$ be a curve in the Euclidean plane. A self-similar motion of $X$ is a map $\hat X: I \times J \rightarrow \mathbf R^{2}$ of the form
\begin{equation*}
\label{hreyfingE}
\hat X(u,t) = g(t)e^{if(t)}X(u) + H(t).
\end{equation*}
Here $J$ is an interval containing $0$ and $f,g: J\rightarrow \mathbf R$ and $H: J \rightarrow \mathbf R^{2}$ are differentiable functions s.t.\ $f(0) = 0$, $g(0) = 1$ and $H(0) = 0$, and hence $\hat X(u,0) = X(u)$. The function $f$ determines the rotation, $g$ determines the dilation and $H$ is the translation term.

The self-similar motion leaves the curve invariant if and only if
\begin{equation*}
\left\langle \frac{\partial \hat X}{\partial t}(u,t),N(u,t) \right\rangle = 0,
\end{equation*}
or equivalently
\begin{equation*}
\label{FullJafnaE}
\begin{aligned}
g(t)f'(t)\langle X(u),T(u) \rangle&+ g'(t)\langle X(u),N(u)\rangle   \\
&+ \langle e^{-if(t)}H'(t),N(u)\rangle = 0,
\end{aligned}
\end{equation*}
for all $u \in I$, $t  \in J$. Here $T=X_s$ is the unit tangent ($s$ is the Euclidean arc-length) and $N = iT$ is the leftward pointing unit normal. By looking at this equation at time $t=0$, we see that $X$ has to satisfy
\begin{equation}
\label{AdalJafnaE}
a\langle X,T \rangle + b \langle X,N \rangle +\langle C,N \rangle = 0,
\end{equation}
where $a=f'(0)$, $b = g'(0)$ and $C = H'(0)$. Satisfying an equation of this form is also a sufficient condition for $X$ to be invariant under a self-similar motion. To see that, we look at three cases.

When $(a,b)=(0,0)$ and $C\neq 0$, Equation \eqref{AdalJafnaE} reduces to $\langle C,N \rangle = 0$, so $X$ is any straight line with direction vector $C$, which is of course invariant under any translation in direction $C$.

When $(a,b)\neq (0,0)$ and $C=0$, we introduce the functions $\tau = \langle X, T \rangle$ and $\nu = \langle X, N \rangle$. Much like their Minkowski counterparts, they always satisfy
\begin{equation*}
\label{TauNuDiffurE}
\begin{aligned}
 \tau_s &= 1 + k\nu \\
\nu_s &= -k\tau,
\end{aligned}
\end{equation*}
where $k = \langle T_s,N \rangle$ is the signed curvature. Here  \eqref{AdalJafnaE} becomes $a\tau + b\nu = 0$. The case $b=0$ yields $\tau = 0$, so $X$ is a circle with center at the origin, which is of course invariant under any rotation around the origin. Now assume $b \neq 0$. Then we note that
$\frac{d}{ds}(b\tau-a\nu) = b$ so $b\tau-a\nu = bs$ and hence $\tau = \frac{b^2}{a^2+b^2}s$, $\nu = -\frac{ab}{a^2+b^2}s$ and $k=\frac{a}{bs}$. Therefore, the curve is given by the formula
\begin{equation*}
X = (\tau + i \nu)e^{i\int kds} = \frac{b}{b+ia}s^{\frac{b+ia}{b}}, \quad s>0.
\end{equation*}
Curves of this form are known in the literature as \emph{logarithmic spirals}  (Figure \ref{logspiral}) and they remain invariant under any combination of rotation and dilation such that $bg(t)f'(t) = ag'(t)$, i.e., $e^{bf(t)} = g(t)^a$. When $a=0$, the spiral reduces to a straight line through origin, which is invariant under any scaling.

\begin{figure}
	\centering
	\begin{subfigure}[t]{0.48\textwidth}
		\centering
		\includegraphics{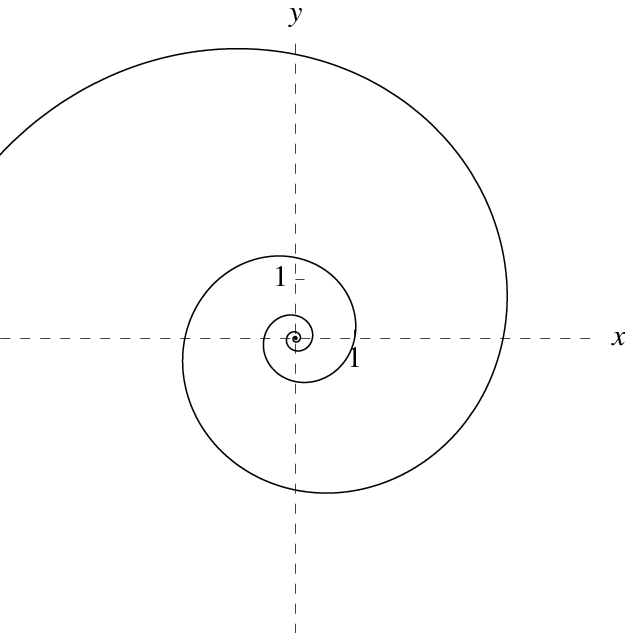}
		\caption{$\frac a b = 5$.}
		\label{logspiral1}
	\end{subfigure}
	\quad
	\begin{subfigure}[t]{0.48\textwidth}
		\centering
		\includegraphics{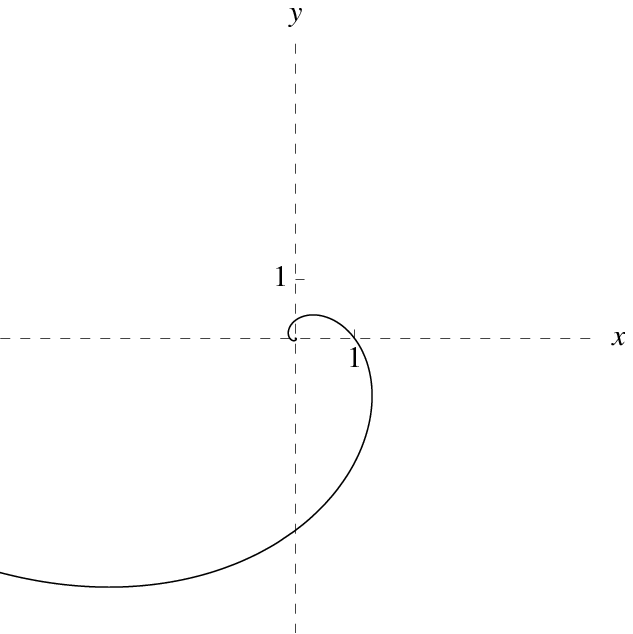}
		\caption{$\frac a b = - \frac 4 3$.}
		\label{logspiral2}
	\end{subfigure}
	\caption{Logarithmic spirals in the Euclidean plane.}
	\label{logspiral}
\end{figure}

When $(a,b)\neq (0,0)$ and $C\neq 0$, we can get rid of the translation term $C$ from \eqref{AdalJafnaE} by translating $X$ by the fixed vector $\frac{C}{b+ia}$. Then we are back in the case above, so $X$ is a logarithmic spiral (circle/line) with center at $-\frac{C}{b+ia}$. 

So we have recovered the following well-known result.

\begin{thm} The only curves in the Euclidean plane that are invariant under a self-similar motion are
\begin{itemize}
\item lines, invariant under translations along their direction vector and dilations about a point on the line,
\item circles, invariant under rotations around their midpoint,
\item logarithmic spirals, $X = \frac{1}{1+i\alpha}s^{1+i\alpha}$, $s>0$, invariant under any combination of rotation and dilation such that $e^{f(t)} = g(t)^\alpha$. 
\end{itemize}
\end{thm}

Now we move to the Minkowski plane. Let $X:I \rightarrow \mathbf R^{1,1}$ be a space-like curve in the Minkowski plane. The self-similar motion given by \eqref{hreyfing} leaves the curve invariant if and only if
\begin{equation*}
\left\langle \frac{\partial \hat X}{\partial t}(u,t),N(u,t) \right\rangle = 0,
\end{equation*}
or equivalently,
\begin{equation*}
\label{FullJafnaM}
\begin{aligned}
g(t)f'(t)\langle X(u),T(u) \rangle&- g'(t)\langle X(u),N(u)\rangle   \\
&- \langle e^{-hf(t)}H'(t),N(u)\rangle = 0,
\end{aligned}
\end{equation*}
which at time $t=0$ results in
\begin{equation}
\label{AdalJafnaM}
a\langle X,T \rangle - b \langle X,N \rangle - \langle C,N \rangle = 0.
\end{equation}
Satisfying an equation of this form is also a sufficient condition for $X$ to be invariant under a self-similar motion. As before, we look at three cases.

When $(a,b)=(0,0)$ and $C\neq 0$, we have $\langle C,N \rangle = 0$ so $X$ is any straight line with direction vector $C$, which is of course invariant under any translation in direction $C$.

Assume $(a,b)\neq(0,0)$ and $C=0$. Then Equation \eqref{AdalJafnaM} reduces to $a\tau - b\nu = 0$. When $b=0$, we have $\tau = 0$ so $X$ is a hyperbola (asymptotic to the light-like axes), which is invariant under any hyperbolic rotation. When $a=0$, we have $\nu = 0$ so $X$ is a line through the origin, which is invariant under any scaling. In the general case, note that $\frac{d}{ds}(b\tau-a\nu) = b$ so $b\tau - a\nu = bs$. This has no solutions when $a^2=b^2$, but when $a^2\neq b^2$ we get $\tau = \frac{b^2}{b^2-a^2}s$, $\nu = \frac{ab}{b^2-a^2}s$ and $k = \frac{a}{bs}$. Therefore, the curve is given by the formula
\begin{equation*}
X = (\tau-h\nu)e^{\int k ds} = \frac{b}{b+ha}s^{\frac{b+ha}{b}}, \quad s>0.
\end{equation*}
In the diagonal basis, it takes the form
\begin{equation*}
\xi = \frac{b}{b+a}s^{\frac{b+a}{b}}, \quad \eta = \frac{b}{b-a}s^{\frac{b-a}{b}}, \quad s>0.
\end{equation*}
These Minkowski analogues of logarithmic spirals are invariant under any combination of hyperbolic rotation and dilation such that $bg(t)f'(t) = ag'(t)$, i.e., $e^{bf(t)} = g(t)^a$. They show different behavior depending on whether $a^2<b^2$ or $a^2>b^2$, as can be seen in Figure \ref{snudur}.

\begin{figure}
	\centering
	\begin{subfigure}{0.48\textwidth}
		\centering
		\includegraphics{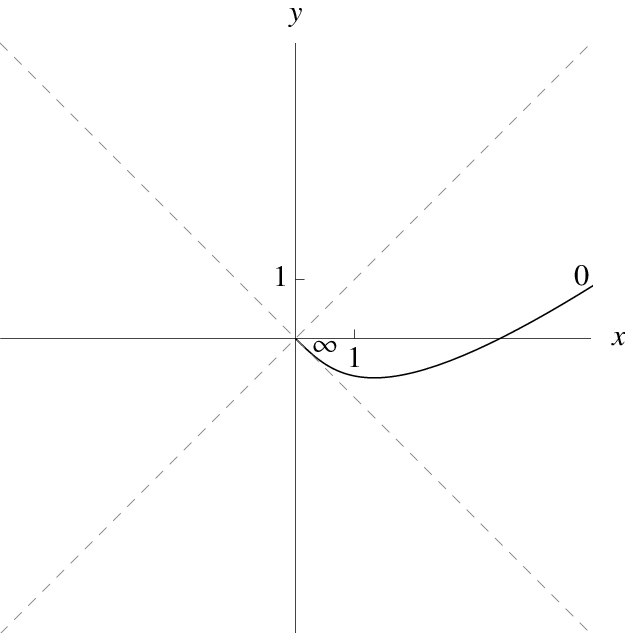}
		\caption{$\frac a b = \frac 1 2$.}
	\end{subfigure}
	\quad
	\begin{subfigure}{0.48\textwidth}
		\centering
		\includegraphics{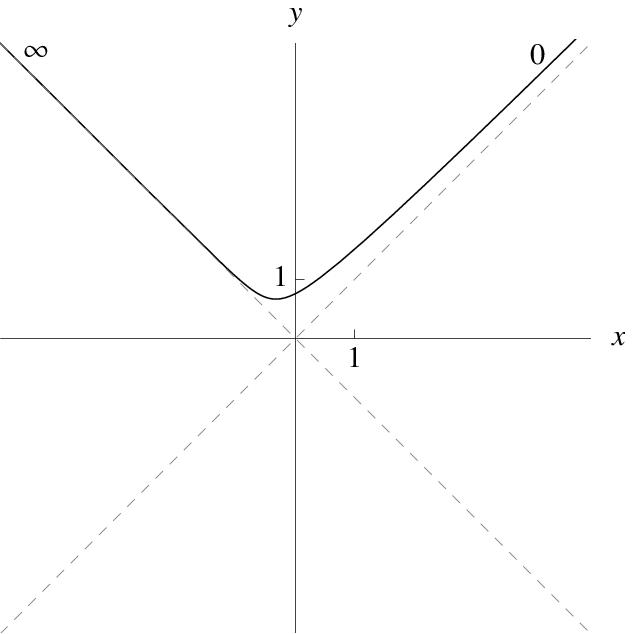}
		\caption{$\frac a b = 2$.}
	\end{subfigure}

	\caption{Minkowski analogues of logarithmic spirals.}
	\label{snudur}
\end{figure}

Assume $(a,b)\neq(0,0)$ and $C\neq0$. When $a^2 \neq b^2$, we can get rid of the translation term $C$ from Equation \eqref{AdalJafnaM} by translating the curve $X$ by the fixed vector $\frac{C}{b+ha}$. Then we are back to the case above, so $X$ is one of those curves centered at $-\frac{C}{b+ha}$. We are left with the case $a^2=b^2$. By reflecting across the $y$-axis if necessary, we may assume $a=b$.  By translating $X$ along the $\xi$-axis, we can cancel out the $\xi$-component of $C$, so we may assume $C$ only has a $\eta$-component, i.e., $C = (0,c)$. By applying a hyperbolic rotation to $X$ and reflecting across the origin if necessary, we can assume $c=b$. Then Equation \eqref{AdalJafnaM} takes the form $\tau - \nu - \langle (0,1), N \rangle = 0$. The equivalent  ODE for $\xi$ as a function of $\eta$ is $\xi'(\eta) = 2\xi (\eta)$, whose solution is $\xi = e^{2\eta}$. This curve is invariant under any combination of hyperbolic rotation, dilation and translation such that
$g(t)=e^{f(t)}$ and $H(t) = (0,f(t))$. Here the screw-dilation factor has the form $g(t)e^{hf(t)} = (e^{2f(t)},1)$, so it is just a dilation in the $\xi$-variable.

So we have proved the following.

\begin{thm}
\label{Minvariant}
 The only curves in the Minkowski plane that are invariant under a self-similar motion are
\begin{itemize}
\item lines, invariant under translations along their direction vector and dilations about a point on the line,
\item hyperbolas with light-like asymptotes, invariant under hyperbolic rotations around their midpoint,
\item Minkowski analogues of logarithmic spirals, $X = \frac{1}{1+h\alpha}s^{1+h\alpha} = \left(\frac{s^{1+\alpha}}{1+\alpha},\frac{s^{1-\alpha}}{1-\alpha}\right)$, $s>0$, invariant under any combination of hyperbolic rotation and dilation such that $e^{f(t)} = g(t)^\alpha$,
\item the curve $\xi = e^{2\eta}$, invariant under any combination of hyperbolic rotation, dilation and translation such that $g(t)=e^{f(t)}$ and $H(t) = (0,f(t))$.
\end{itemize}
\end{thm}

\section*{Acknowledgements}
I would like to thank my friend Eric Marberg for reading over the draft and providing valuable feedback. I would also like to thank my advisor, Tobias Colding, for guidance and support.

\bibliographystyle{amsplain}
\bibliography{SSMinkowskiMai}

\end{document}